\documentclass[11pt]{article}
\usepackage[margin=1.15in]{geometry}
\usepackage{amsmath,amssymb,amsthm}
\usepackage{booktabs}
\usepackage{graphicx}
\usepackage{tikz}
\usetikzlibrary{arrows.meta,positioning,calc}
\usepackage{enumitem}
\usepackage[colorlinks=true,linkcolor=blue!50!black,citecolor=blue!50!black,urlcolor=blue!50!black]{hyperref}

\newtheorem{theorem}{Theorem}[section]
\newtheorem{lemma}[theorem]{Lemma}
\newtheorem{proposition}[theorem]{Proposition}

\newtheorem{conjecture}[theorem]{Conjecture}
\theoremstyle{definition}
\newtheorem{definition}[theorem]{Definition}
\newtheorem{example}[theorem]{Example}
\newtheorem{problem}[theorem]{Open Problem}
\theoremstyle{remark}
\newtheorem{remark}[theorem]{Remark}

\newenvironment{musical}
  {\par\medskip\noindent\begin{quote}\itshape\textbf{\upshape In musical terms.}\ }
  {\end{quote}\medskip}

\newcommand{\Z}{\mathbb{Z}}
\newcommand{\ec}{\mathrm{ec}}
\newcommand{\form}[1]{\ensuremath{\mathrm{#1}}}

\newif\ifanon
\anonfalse            

\title{The Missing Zigzag:\\
Cycles of Semitone Trichords and a Conservation Law\\ in Equal Temperament\thanks{\ifanon\else%
See the acknowledgment for the collaboration under which this work was
carried out, and Appendix~\ref{app:lean} for the machine-checked
formalization.  Code and the Lean development accompany the paper as
ancillary files.\fi}}
\ifanon
  \author{}
\else
  \author{David Victor Feldman\\
    \small University of New Hampshire\\
    \small\texttt{dvfinnh@gmail.com}}
\fi
\date{August 2026}
\newcommand{\dedication}{\emph{In memory of Larry Polansky (1954--2024)}}

\begin{document}
\maketitle
\ifanon\else\begin{center}\dedication\end{center}\fi

\begin{abstract}
\noindent
A cycle of nine pitch classes in twelve-tone equal temperament can be
arranged so that its nine consecutive three-note windows realize,
exactly once each, every trichord type containing a semitone---types
taken up to transposition but not inversion.  Such cycles generalize to
$n$-tone equal temperament, where the windows realize the $n-3$
semitone-containing types.  In every one of the $1764$ such cycles in
twelve-tone equal temperament the chromatic trichord $012$ appears
\emph{scalarly} (stepwise), as a direct chromatic run, and never in a broken
``zigzag'' contour such as $1$--$0$--$2$.  Why is the zigzag missing?

We show that the zigzag presentation behaves as a conserved
charge of an underlying interval-flow network: it occurs if and only if
$3\mid n$ and $n\notin\{6,12\}$.  The forward obstruction is a
$\Z/3$-valued conservation law, proved for all $n$ by a weight-function
argument; the two exceptional temperaments, one of them the common
Western tuning, are excluded by finite certificates; and existence for
every other multiple of three follows from an explicit construction.  We
quantify rarity, and give exact censuses beyond the reach of direct
search.  The classification is machine checked in Lean~4 in both
directions; the finite twelve-tone exclusion within it rests on a
compiled finite computation.  The paper is written for two audiences.
\end{abstract}



\tableofcontents

\part{The musical question}

\section{Introduction}\label{sec:intro}

Consider pitch-class cycles with the following exhaustive property: as one
traverses the cycle, the successive three-note windows spell out, once
each, \emph{every} trichord type that contains a semitone.  The condition
arises naturally in compositional practice, and the cycles are usable
directly as compositional material (Section~\ref{sec:composition}).  In
twelve-tone equal temperament ($12$TET) there are nine such types---in
set-class shorthand, $012, 013, 014, \dots, 01\mathrm{T}$, where types are
counted up to transposition only, so that $013$ and its inversion
$01\mathrm{T}$ count separately%
\footnote{Throughout, $\mathrm{T}=10$ and $\mathrm{E}=11$.  Our convention
of transpositional (not inversional) equivalence is essential: the
phenomenon studied here is invisible at the level of Forte set classes,
which conflate a type with its inversion.}.
A cycle of nine pitch classes has nine windows, so the requirement is
exactly met: each window one type, no type repeated.  Such cycles exist;
here is one, starting on C:
\[
\textrm{C}\quad \textrm{C}\sharp\quad \textrm{D}\quad \textrm{E}\quad
\textrm{F}\quad \textrm{C}\quad \textrm{B}\quad \textrm{A}\flat\quad
\textrm{G}\quad (\textrm{back to C}),
\]
that is, the pitch-class cycle $(0,1,2,4,5,0,11,8,7)$.  Note that pitch
classes may repeat around the cycle (C occurs twice); only the trichord
types must not.

The chromatic
trichord $012$ must occur as some window---three consecutive notes whose
classes form $\{x,x+1,x+2\}$.  A window has an \emph{ordering}, and up to
the natural symmetries (transposition, running the cycle backwards,
inverting it) there are exactly two essentially different ways the
chromatic trichord can be laid down in time: \emph{scalarly}, as a direct
run $x,\,x{+}1,\,x{+}2$ (or its descent), or as a \emph{zigzag}---any of
the broken contours $x,\,x{+}2,\,x{+}1$, or $x{+}1,\,x,\,x{+}2$, and so
on.  In the example above the chromatic window is C--C$\sharp$--D: scalar.
Both contours might be expected to occur, with the zigzag---four orderings
against the scalar's two---if anything favoured.

They do not both occur.  \emph{Every one} of the $1764$ cycles in $12$TET
(with starting pitch fixed; $52$ cycles up to all symmetries) presents the
chromatic scalarly.  The zigzag is missing.  We call this the
\emph{Missing Zigzag Puzzle} (MZP).  The puzzle sharpens when one varies
the temperament: in $n$TET the same construction asks for cycles of $n-3$
pitch classes realizing the $n-3$ semitone-containing trichord types
$\{0,1,k\}$, $2\le k\le n-2$, and the zigzag's presence depends on $n$ in a
way that no analysis conducted at the level of local constraints can
detect (Proposition~\ref{prop:local}).

The main results, in plain language:

\begin{enumerate}[leftmargin=2em]
\item \textbf{Classification} (Theorem~\ref{thm:classification}).  Zigzag
cycles exist in $n$TET if and only if $n$ is a multiple of $3$ and
$n\notin\{6,12\}$.  Thus $9$TET, $15$TET, $18$TET, $21$TET, \dots all
admit zigzags---indeed $9$TET admits \emph{only} zigzags---while
twelve-tone equal temperament is one of exactly two multiples of three
where the zigzag vanishes, and it vanishes there for a subtler reason than
in, say, $10$TET or $14$TET.
\item \textbf{Separation}.  A reduction (Theorem~\ref{thm:reduction})
splits a cycle into the choice of contour for each trichord type and the
temporal order in which those contours are realized.  The obstruction
occurs entirely in the first layer; the second contributes only the
realization and its multiplicity.
\item \textbf{Mechanism}.  The chromatic type's presentation is a conserved
charge of an interval-flow network attached to any cycle: scalar
presentations are \emph{neutral} (they are the system's only
```self-loops'''), zigzag presentations inject a unit of charge which,
provably, can never be neutralized and must instead propagate around the
entire cycle and cancel arithmetically---possible only when $3\mid n$
(Theorem~\ref{thm:mod3}).
\item \textbf{Rarity}.  Even where zigzags exist they are exponentially
rare: the transfer dynamics splits into disjoint neutral and charged
sectors with growth constants $6.3723$ and $3$ per step
(Section~\ref{sec:dynamics}), so that by $36$TET zigzag configurations
are outnumbered roughly $1.8$ million to one.
\item \textbf{Census}.  Exact counts of cycles by chromatic presentation,
extending far beyond direct search (Section~\ref{sec:census}); e.g.\
$B(24)=952{,}560$ and $B(27)=30{,}352{,}896$ zigzag cycles.
\end{enumerate}

\subsection{How to read this paper}\label{sec:guide}

The paper serves two audiences and is structured so that each has a
through-line.

\emph{Readers primarily interested in the music} can read
Part~I in full (no formal prerequisites), then the statements---not the
proofs---of the theorems in Part~II, each of which is followed by a gloss
marked \textbf{In musical terms}, and then Part~IV, the compositional
epilogue.  The censuses in Section~\ref{sec:census} are readable as
tables.

\emph{Mathematical readers} may skim Part~I for the definitions-by-example,
then read Parts~II and~III linearly.  The ingredients are
elementary---directed multigraphs, Eulerian circuits, modular
arithmetic---but the classification appears to require the global
structures introduced below.  Proposition~\ref{prop:local} shows that the
obstruction is invisible to the adjacency data alone, which is why the
argument is conducted at the level of whole closed cycles.  Connections to the theory of
universal cycles \cite{CDG1992,Jackson1993,Hurlbert1994} are discussed in
Section~\ref{sec:setup}.

Several results below rest on finite computations: certificates, gadget
verifications, and exhaustive analyses of a $25$-state automaton.  Each is
finite and reproducible, and plays the role of a case check inside an
otherwise conventional proof.  Every enumerative claim and every explicit example in the paper has been
computed by two independent programs using different representations of
trichord classes.  The entries of Table~\ref{tab:census} for
$8\le n\le 18$ were recomputed by direct enumeration of pitch words; the
entries for $n=21,24,27$, which lie beyond direct search, were recomputed
from the BEST theorem \cite{AEdB1951} by exact matrix-tree determinants,
a method sharing no code with the circuit enumeration that produced them,
and the two agree plan by plan wherever both are available.  Statements are labelled as theorems, as
finite verifications, or as conjectures, and the distinction is maintained
throughout.  Source code accompanies the paper; the principal results of
Part~II are additionally machine checked in Lean~4
(Appendix~\ref{app:lean}).

\section{The cycles, by example}\label{sec:examples}

\subsection{Twelve-tone cycles}\label{sec:12tet}

Fix $12$TET and identify pitch classes with $\Z_{12}=\{0,1,\dots,11\}$,
$0=\textrm{C}$.  The trichord types containing a semitone are
\[
\{0,1,k\}, \qquad k=2,3,\dots,10,
\]
nine types, transpositionally but not inversionally reduced; note
$\{0,1,11\}$ is a transposition of $\{0,1,2\}$, which is why $k$ stops at
$n-2=10$.

\begin{definition}\label{def:cycle-informal}
A \emph{trichord cycle} in $12$TET is a cyclic sequence of nine pitch
classes $c_0,c_1,\dots,c_8$, repetitions of pitch class allowed, such that
the nine windows $\{c_i,c_{i+1},c_{i+2}\}$ (indices mod~$9$) realize
the nine types $\{0,1,k\}$, for $k=2,\dots,10$, each exactly once.
\end{definition}

The example of the introduction, $(0,1,2,4,5,0,11,8,7)$, has windows
\[
\underbrace{012}_{k=2},\;
\underbrace{124}_{k=3},\;
\underbrace{245}_{k=10},\;
\underbrace{450}_{k=8},\;
\underbrace{5\,0\,11}_{k=6},\;
\underbrace{0\,11\,8}_{k=9},\;
\underbrace{11\,8\,7}_{k=4},\;
\underbrace{870}_{k=5},\;
\underbrace{701}_{k=7},
\]
each type once.

\begin{figure}[ht]
\centering
\includegraphics[width=1.00\textwidth]{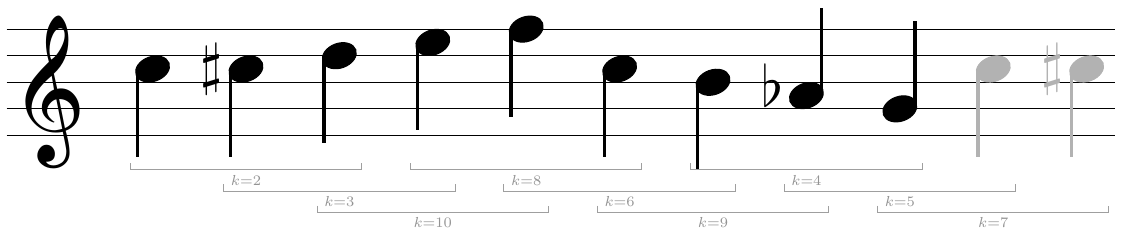}
\caption{The twelve-tone cycle $(0,1,2,4,5,0,11,8,7)$ in notation.  The
final two notes are repeated in grey to display the two windows that wrap
around the cycle; brackets mark the nine windows and name the class
$\{0,1,k\}$ each realizes.  The chromatic window is the opening
C--C$\sharp$--D, a scalar presentation.}
\label{fig:cycle12}
\end{figure}

Two counting conventions will be used throughout.  The \emph{labelled
count} fixes $c_0=0$ and counts sequences; this quotients by transposition
only and is the natural count for the transfer-matrix methods below.  The
\emph{orbit count} further quotients by rotation of the cycle, retrograde
(reversal), and inversion ($x\mapsto -x$).  In $12$TET there are $1764$
labelled cycles and $52$ orbits.

\subsection{Scalar and zigzag}\label{sec:scalar-zigzag}

The chromatic window---the three consecutive notes realizing
$\{x,x{+}1,x{+}2\}$---appears in one of six orderings.  Up to the cycle
symmetries these collapse to two contours:
\[
\underbrace{(x,\,x{+}1,\,x{+}2)\ \text{ or }\ (x{+}2,\,x{+}1,\,x)}_{\text{\emph{scalar}: a chromatic run}}
\qquad\qquad
\underbrace{(x,\,x{+}2,\,x{+}1),\ (x{+}1,\,x,\,x{+}2),\ \dots}_{\text{\emph{zigzag}: four broken contours}}
\]
A cycle is called \emph{scalar} or \emph{zigzag} according to the contour
of its chromatic window.  (Every cycle has exactly one chromatic window,
since the type $\{0,1,2\}$ occurs exactly once.)

\begin{musical}
A scalar chromatic is a direct chromatic run---C, C$\sharp$, D---heard as a
single gesture.  A zigzag chromatic breaks the run: C$\sharp$, C, D, say,
or C, D, C$\sharp$: the same three pitch classes, but with a changed
contour, the semitone concealed inside a leap-and-return figure.  The
puzzle is that in $12$TET the second kind of cycle simply does not exist.
\end{musical}

\subsection{The census, first look}\label{sec:census-first}

Write $A(n)$ and $B(n)$ for the labelled counts of scalar and zigzag cycles
in $n$TET (so cycles have length $n-3$ and realize the types $\{0,1,k\}$,
$2\le k\le n-2$).  Direct enumeration for small $n$ gives
Table~\ref{tab:first-census}.

\begin{table}[ht]
\centering
\begin{tabular}{lrrrrrrrrrrrrrr}
\toprule
$n$ & 6 & 7 & 8 & 9 & 10 & 11 & 12 & 13 & 14 & 15 & 16 & 17 \\
\midrule
scalar $A(n)$ & 0 & 0 & 30 & 0 & 84 & 0 & 1764 & 320 & 24596 & 10368 & 440830 & 159656 \\
zigzag $B(n)$ & 0 & 0 & 0 & 24 & 0 & 0 & 0 & 0 & 0 & 576 & 0 & 0 \\
\midrule
orbits (total) & 0 & 0 & 2 & 1 & 4 & 0 & 52 & 8 & 562 & 228 & 8489 & 2853 \\
\bottomrule
\end{tabular}
\caption{Labelled counts of trichord cycles by chromatic contour, and total
orbit counts, for small $n$.  Three temperaments ($n=6,7,11$) admit no
cycles at all.}
\label{tab:first-census}
\end{table}

Three phenomena leap out.  First, the zigzag column is almost all zeros;
where it is nonzero ($n=9,15$), $n$ is a multiple of $3$.  Second, at
$n=9$ the situation \emph{reverses}: all $24$ cycles are zigzag, and there
are no scalar cycles at all.  Third---beyond the horizon of
Table~\ref{tab:first-census}---the multiples of three are not uniform:
$n=18$ and $n=21$ have zigzags ($B(18)=1800$, $B(21)=83{,}232$) while
$n=12$, alone with $n=6$ among all multiples of three, has none.  The
common temperament, twelve-tone equal temperament, is exceptional.

\begin{example}[The $9$TET miniature]\label{ex:9tet}
$9$TET admits exactly one cycle up to all symmetries:
\[
(0,\,1,\,4,\,3,\,8,\,2),
\]
six pitch classes ($133$-cent steps), whose six windows realize the six
types $\{0,1,k\}$, $2\le k\le 7$, and whose chromatic window is the
\emph{valley} $8,2,\dots$ realizing $\{8,0,1\}$ in zigzag contour.  In the
smallest temperament that admits these cycles at all, the zigzag is not
merely possible but forced.
\end{example}

\begin{example}[A $15$TET zigzag]\label{ex:15tet}
In $15$TET ($80$-cent steps) the cycle
\[
(0,\,2,\,1,\,6,\,0,\,14,\,9,\,13,\,10,\,11,\,4,\,3)
\]
opens with the zigzag chromatic $0,2,1$ and realizes the twelve types
in the window order
\[k=2,\,5,\,6,\,7,\,10,\,11,\,4,\,3,\,9,\,8,\,12,\,13.\]

\begin{figure}[ht]
\centering
\includegraphics[width=0.66\textwidth]{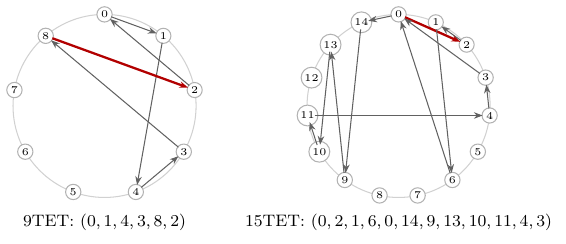}
\caption{The microtonal examples, drawn on the pitch-class circle rather
than a staff: steps of $133$ and $80$ cents have no standard notation.
Left, the unique $9$TET cycle; right, a $15$TET cycle.  The arrow marks
the chromatic window, zigzag in both cases.}
\label{fig:micro}
\end{figure}
\end{example}

\section{Why the puzzle resists local reasoning}\label{sec:local}

Before developing the machinery that solves the MZP, we explain---in
terms a musical reader can follow completely---why it is genuinely hard,
and in particular why no analysis of \emph{local} constraints (which
trichord forms can follow which) can ever solve it.  This section also
proves the one fact about the puzzle that \emph{is} local
(Lemma~\ref{lem:adjacent}), a fact of independent compositional interest.

\subsection{Hinges and joints}\label{sec:hinges}

Consecutive windows of a cycle overlap in two notes, hence share one
melodic interval.  What that shared interval is matters enormously:

\begin{itemize}[leftmargin=2em]
\item If the shared interval is a \emph{semitone} ($\pm1$), it constrains
almost nothing: every one of the $n-3$ types contains a semitone, so a
semitone step is a free \emph{hinge} between windows.
\item If the shared interval has interval class $g\ge 2$, at most four
types contain it---exactly those $\{0,1,k\}$ with
$k\in\{g,\,g{+}1,\,n{-}g,\,n{-}g{+}1\}$ that lie in the range
$2\le k\le n-2$---so a larger shared interval is a rigid \emph{joint}
pinning both neighbours into a set of size at most four: two types
adjacent in the natural enumeration, together with their inversions.
The count is exactly four for interior $g$, and drops at the extremes:
three when $g=2$ (where $n-g+1$ falls outside the range), and two or
three at the largest $g$, where the two pairs coincide.
\end{itemize}

The scarcest joint is the \emph{whole tone} ($g=2$): only the types $k=2$
(the chromatic itself), $k=3$, and $k=n-2$ contain it.  Since a zigzag
chromatic window, by its contour, contains a whole-tone \emph{step}, a
zigzag must spend a whole-tone joint on the chromatic---consuming the
rarest resource in the system.  This observation, in embryo, is the germ
of everything that follows; but as stated it proves nothing, for reasons
we now make precise.

There is one exact and useful bookkeeping identity at this level.  Each
type contains one semitone pair, except the chromatic, which contains two;
in a cycle each semitone pair is realized either as an adjacent melodic
step (each such step serving two windows) or as the outer ```skip''' of a
window.  Writing $S$ for the number of semitone steps and $K$ for the
number of windows whose skip is a semitone,
\begin{equation}\label{eq:2SK}
2S+K \;=\; n-2 .
\end{equation}

\subsection{The chromatic cadence lemma}\label{sec:cadence}

\begin{lemma}\label{lem:adjacent}
In any trichord cycle, two consecutive semitone steps must have the same
direction, and the window they span is the chromatic window presented
scalarly.  Consequently a scalar cycle contains exactly one adjacent pair
of semitone steps, and a zigzag cycle contains none.
\end{lemma}

\begin{proof}
If consecutive steps are $+1,-1$ or $-1,+1$, the window's outer notes
coincide, so the window has only two distinct pitch classes and realizes
no trichord type: impossible, since every window must realize a type.
Same-signed consecutive semitones give the window $\{x,x\pm1,x\pm2\}$,
the chromatic, presented scalarly; and the chromatic occurs once.
\end{proof}

\begin{musical}
Every $12$TET trichord cycle---and, we will see, every cycle in any
temperament where zigzags are impossible---contains exactly one direct
chromatic run, and chromatic motion in consecutive steps happens
\emph{nowhere else} in the cycle.  The chromatic run is thus a structural
landmark, a unique built-in cadence point that a composition can
articulate, conceal, or orbit.
\end{musical}

\subsection{Locality does not decide the question}\label{sec:locality-fails}

The natural local approach to the MZP is the graph of \emph{local
adjacencies}: nodes the $6(n-3)$ ordered forms of the types, edges joining
forms that can legally overlap in a cycle.  If the zigzag forms of the
chromatic were isolated in that graph---unable to attach to any
neighbour---the MZP would be settled locally.  They are not.

\begin{proposition}\label{prop:local}
In $12$TET the zigzag forms of the chromatic trichord admit five directed
local adjacencies: through whole-tone steps to forms of the types $k=3$
and $k=10$, exactly as the arithmetic of Section~\ref{sec:hinges}
permits.  No completed cycle realizes any of the five.
\end{proposition}

\begin{proof}
The five adjacencies are read off the arithmetic of
Section~\ref{sec:hinges}.  For the second assertion, compare the six
forms of the chromatic class: the two scalar forms are $(1,1)$ and
$(-1,-1)$, and each of the four zigzag forms $(2,-1)$, $(-1,2)$,
$(1,-2)$, $(-2,1)$ carries exactly one entry of interval class $2$.  A
chromatic window therefore has a whole-tone step if and only if it is
zigzag.  An adjacency of the stated kind is a whole-tone step at the
chromatic window, hence a zigzag chromatic, which
Theorem~\ref{thm:classification} excludes at $n=12$.
\end{proof}

\noindent One of the five is excluded for a visible reason---the chromatic
cannot neighbour itself, since it occurs once.  The remaining four are
excluded by no local condition.

Proposition~\ref{prop:local} fixes what an explanation of the MZP must
look like: an invariant of whole closed cycles, not a constraint on
admissible overlaps.  (The adjacency graph harvested \emph{from} the
solved cycles does exhibit an isolated chromatic, but that graph is a
restatement of the phenomenon and cannot explain it.)  Part~II constructs
the required invariant; Section~\ref{sec:dynamics} identifies it as a
conservation law.

\part{The mathematics}

\section{Formal setup}\label{sec:setup}

Fix $n\ge 6$ and work in $\Z_n$.  For $2\le k\le n-2$ let $T_k$ denote the
transposition class of the set $\{0,1,k\}\subset\Z_n$; these are exactly
the $n-3$ trichord classes containing interval class $1$
(with $\{0,1,n-1\}=T_2$, the chromatic class).

\begin{definition}\label{def:cycle}
A \emph{trichord cycle} (henceforth \emph{cycle}) is a cyclic word
$c_0c_1\cdots c_{m-1}$ over $\Z_n$, $m=n-3$, such that the windows
$W_i=\{c_i,c_{i+1},c_{i+2}\}$ (indices mod $m$) are three-element sets
realizing each class $T_k$, $2\le k\le n-2$, exactly once.  A cycle is
\emph{labelled} if $c_0=0$.  The chromatic window is the $W_i$ realizing
$T_2$; the cycle is \emph{scalar} if that window's ordered triple is a
translate of $(0,1,2)$ or $(2,1,0)$, and \emph{zigzag} otherwise.  Let
$A(n)$, $B(n)$ count labelled scalar, zigzag cycles.
\end{definition}

In the language of Chung, Diaconis and Graham \cite{CDG1992} (see also
\cite{Knuth2011} for the de~Bruijn background), a cycle is a
\emph{universal cycle}, of the shortest possible length, for the family
$\{T_k\}$ read through a sliding window of width $3$.  The u-cycle
literature \cite{CDG1992,Jackson1993,Hurlbert1994} establishes that
existence of universal cycles typically hinges on global divisibility
conditions invisible to local transition rules; the present paper may be
read as a case study in which not just existence but a \emph{presentation
refinement}---how one distinguished element of the family is realized---is
governed by such a condition.  We have not located the refined question
in that literature.

The symmetry group acting on cycles is generated by transposition
($c_i\mapsto c_i+t$), rotation ($c_i \mapsto c_{i+r}$), retrograde
($c_i\mapsto c_{-i}$) and inversion ($c_i\mapsto -c_i$); scalar and zigzag
are unions of orbits, being defined by the two orbits of orderings of
$T_2$ under retrograde--inversion.

\section{The Reduction Theorem}\label{sec:reduction}

The engine of the paper converts cycles---objects with both harmonic
content (which classes) and temporal order (which Eulerian
arrangement)---into pairs of simpler objects in which the two aspects
separate.

Each window is a translate of one of the six orderings of $\{0,1,k\}$;
recording only the ordered pair of melodic steps $(a,b)$ (so the window is
$c,\,c{+}a,\,c{+}a{+}b$), the six \emph{forms} of class $T_k$ are:
\begin{equation}\label{eq:forms}
\begin{array}{lll}
\form{I}:  (1,\,k-1) & \form{II}: (k,\,1-k) & \form{III}: (-1,\,k)\\[2pt]
\form{IV}: (k-1,\,-k) & \form{V}:  (-k,\,1) & \form{VI}: (1-k,\,-1)
\end{array}
\qquad(\text{entries in }\Z_n).
\end{equation}
For $k=2$: forms $\form{I}=(1,1)$ and $\form{VI}=(-1,-1)$ are the scalar
presentations; $\form{II},\form{III},\form{IV},\form{V}$ are the zigzags.

View each form as a directed edge $a\to b$ on the node set
$\Z_n\setminus\{0\}$ of step values.  Consecutive windows share their
middle step, so a cycle traverses its chosen forms as a closed chain in
which each edge's head matches the next edge's tail: an Eulerian circuit.

\begin{definition}\label{def:choice}
A \emph{choice function} is a map $F$ assigning to each class $T_k$ one of
its six forms.  Write $G_F$ for the directed multigraph of chosen edges
(distinct classes give distinct ordered pairs, so in fact no two edges
share both endpoints).  $F$ is \emph{balanced} if every node of $G_F$ has
in-degree equal to out-degree; \emph{closed} if
$\sum_k a(F(k))\equiv 0 \pmod n$; \emph{weakly connected} if $G_F$ is
connected when edge directions are ignored; and \emph{valid} if balanced,
closed and weakly connected.  Balance together with weak connectivity is
what yields an Eulerian circuit.
\end{definition}

\begin{theorem}[Reduction]\label{thm:reduction}
Labelled cycles correspond bijectively to pairs consisting of a valid
choice function $F$ together with an Eulerian circuit of $G_F$ with a
distinguished starting edge-position.  Hence
\[
A(n)+B(n)\;=\;\sum_{F\ \mathrm{valid}} m\cdot \ec(G_F),
\]
where $\ec(G_F)$ counts Eulerian circuits of $G_F$ \emph{modulo cyclic
rotation of the edge sequence} and the factor $m$ chooses the
distinguished starting edge; the scalar/zigzag split is read off from
$F(T_2)$.
\end{theorem}

\begin{proof}
Given a labelled cycle, each window determines the form of its class and
the chain condition holds by the shared middle step; closure is
$\sum_i (c_{i+1}-c_i)\equiv 0$; balance and connectivity hold because the
edge sequence is a closed trail through all chosen edges.  Conversely, an
Eulerian circuit $(e_1,\dots,e_m)$ of a valid $F$, read as a step sequence
$d_i=a(e_i)$, integrates (from $c_0=0$) to a labelled cycle: each window
realizes the class of its edge, since an ordered step pair determines the
underlying set, and hence the class, uniquely; classes are distinct
because $F$ chooses one edge per class; and $\sum d_i \equiv 0$ is closure.
Distinct arrangements give distinct step sequences (edges are distinct as
an ordered pair $(a,b)$ determines the set $\{0,a,a+b\}$, hence its
transposition class, so no two chosen edges belonging to different
classes can coincide), hence distinct labelled cycles; the $m$ edges of $G_F$ are pairwise distinct, so a
cyclic ordering of them admits exactly $m$ choices of starting edge, and
each choice integrates from the fixed $c_0=0$ to one labelled cycle;
hence the factor $m$.  \emph{(Computationally re-verified: the right-hand
side reproduces every labelled count of Table~\ref{tab:first-census} and
all later censuses.)}
\end{proof}

\begin{musical}
The theorem separates harmony from melody.  A valid choice function is a
\emph{harmonic plan}: a decision, for every trichord type at once, of the
contour in which it will appear---subject to a bookkeeping constraint
(balance) that says every melodic interval must be entered as often as it
is left, plus the requirement that the intervals add up to a whole number
of octaves (closure).  The Eulerian circuit is the \emph{melodic
realization}: an order in which to thread the planned contours.  All of
the obstruction underlying the Missing Zigzag Puzzle lives entirely in
the harmonic plan; the melody enters only through the existence and
multiplicity of Eulerian realizations, which Section~\ref{sec:crystal}
measures.  None of the obstruction lives in the
melody.
\end{musical}

\subsection{Drift: interval cycles that do not close}\label{sec:drift}

A natural question---one a careful reader will ask---is what becomes of
cyclic \emph{interval} patterns satisfying the window conditions whose
steps do not sum to $0$ modulo $n$, and which therefore never close as
pitch-class cycles of length $m$.  In the present framework these are
exactly the objects obtained by relaxing closure: define a
\emph{drift-$s$ cycle} to be a cyclic word of steps $d_0\cdots d_{m-1}$
whose consecutive pairs are forms of the $n-3$ classes, each once, with
$\sum d_i\equiv s \pmod n$.  Theorem~\ref{thm:reduction} holds verbatim
with closure replaced by the drift condition.  Integrating a drift-$s$
word yields a pitch sequence that repeats itself transposed by $s$, and
closes as a genuine pitch-class cycle only after $n/\gcd(s,n)$
periods---of length $m\cdot n/\gcd(s,n)$, realizing each trichord type
once per transposition level $0,s,2s,\dots$.  Drift-$s$ cycles are thus
\emph{sequences} in the time-honored sense: self-transposing patterns.

Two consequences deserve emphasis.  First, the proof of
Theorem~\ref{thm:mod3} invokes balance only, never closure; hence
\emph{zigzag interval cycles of any drift whatsoever exist only when
$3\mid n$}---the conservation law governs the entire drift family, and
the searches confirm it (no balanced zigzag plan exists at
$n=8,10,11,13,14$ for any drift).  Second, at the exceptional
temperaments the zigzag is not absent but \emph{unclosable}.  The full
drift census at $n=12$:

\begin{center}
\begin{tabular}{lccccccc}
\toprule
drift $s$ (and $12{-}s$) & 0 & 1 & 2 & 3 & 4 & 5 & 6\\
\midrule
scalar cycles & 1764 & 1944 & 1863 & 1917 & 1944 & 1971 & 1818\\
zigzag cycles & \textbf{0} & 36 & 36 & 54 & 36 & 54 & 72\\
\bottomrule
\end{tabular}
\end{center}

\noindent Zigzag interval patterns exist at \emph{every} nonzero drift
and vanish exactly at drift zero: of the $48$ balanced zigzag plans, four
are disconnected, and the remaining $44$ distribute as exactly four
realized plans at each of the eleven nonzero drifts, yielding $504$
labelled interval cycles in all.  This is the sharpest statement of the Missing Zigzag Puzzle:
$12$TET's interval grammar contains the zigzag freely; what it forbids
is closure.  For instance, the drift-$1$ word
\[
(2,\,11,\,8,\,1,\,9,\,4,\,7,\,6,\,1)
\]
integrates from $0$ to the pitches
$0,2,1,9,10,7,11,6,0$, then repeats a semitone higher, and higher again,
closing after twelve periods into a $108$-note pitch cycle whose $108$
windows realize each of the nine trichord types at each of the twelve
transposition levels exactly once---with the chromatic broken in every
one of its twelve appearances.

\begin{musical}
The referee's question has a composer's answer.  In twelve-tone equal
temperament you \emph{may} have your zigzag---the broken chromatic
threading a complete tour of the semitone trichords---but only as a
sequence: a pattern that will not return home, rising (or falling) by a
fixed transposition each time around, and closing the circle only after
visiting every key.  The theorems of this paper measure exactly the
price of the broken chromatic: it costs you your tonic.
\end{musical}

\section{The conservation law: zigzag forces $3\mid n$}\label{sec:mod3}

Balance has the following consequence:
if $F$ is balanced then the multisets $\{a(F(k))\}$ and $\{b(F(k))\}$
coincide, so for \emph{any} weight $w:\Z_n\setminus\{0\}\to \Z/3$,
\begin{equation}\label{eq:ledger}
\sum_k \bigl(w(b(F(k)))-w(a(F(k)))\bigr) \;=\; 0 .
\end{equation}
The art is to choose $w$ so that each class's contribution is independent
of the chosen form---so that \eqref{eq:ledger} becomes a fixed identity---%
except at the chromatic, where scalar and zigzag contribute differently.

\begin{theorem}[Conservation law]\label{thm:mod3}
If any cycle presents the chromatic class in zigzag, then $3\mid n$.
\end{theorem}

\begin{proof}
By Theorem~\ref{thm:reduction} it suffices to show no \emph{balanced}
choice function has a zigzag chromatic when $3\nmid n$.  Define
$w:\Z_n\setminus\{0\}\to\Z/3$ by
\[
w(1)=0,\qquad w(-1)=1,\qquad w(j)=n+j-1 \bmod 3 \quad (2\le j\le n-2).
\]
A direct check of the six forms \eqref{eq:forms} of each class $k\ge 3$
gives, for every form,
\[
w(b)-w(a)\;\equiv\; n+k-2 \pmod 3,
\]
independent of the form.  (For instance, form $\form{IV}=(k{-}1,-k)$:
$w(n{-}k)-w(k{-}1) = (2n-k-1)-(n+k-2) = n-2k+1 \equiv n+k+1 \equiv n+k-2$.)
For the chromatic $k=2$, the scalar forms contribute $0$ and the four
zigzag forms contribute $\pm n \bmod 3$.  Summing \eqref{eq:ledger}:
\[
0 \;=\; \sum_{k=3}^{n-2}(n+k-2) \;+\; \chi
\;=\;\Bigl[(n-4)(n-2)+\tfrac{(n-2)(n-1)}{2}-3\Bigr]+\chi \pmod 3,
\]
and the bracket is $\equiv 0$ for every $n$ (check $n\equiv 0,1,2$
separately).  Hence $\chi\equiv 0$: for scalar, $\chi=0$ and there is no
constraint; for zigzag, $\chi=\pm n$, forcing $n\equiv 0 \pmod 3$.
\end{proof}

\begin{musical}
Balance is a ledger: every interval-entry must be matched by an
interval-exit.  The weight $w$ is a bookkeeping currency, a
three-valued colouring of the melodic intervals, chosen so that
every trichord type pays a \emph{fixed} fee into the ledger no matter
which contour it takes---every type except the chromatic, whose scalar
contours pay nothing and whose zigzag contours pay a fee equal to $n$
itself, modulo $3$.  Since the ledger must balance, a zigzag is affordable
only when $n$ is a multiple of $3$.  The chromatic trichord, alone among
all the types, \emph{detects the temperament's divisibility by three}.
Note what the proof never mentions: the order of events.  It is an
accounting identity over the whole harmonic plan, which is why no local,
note-to-note analysis could find it.
\end{musical}

\begin{remark}
The weight is not conjured from nothing: requiring form-independence for
all $k\ge 3$ forces $w$ to be an arithmetic progression on $[2,n-2]$, and
the chromatic is then precisely the class at which the progression's
extrapolated value ($n$) disagrees with the scalar contribution ($0$).
The chromatic is the unique ```defect''' of the progression, and the defect
size is $n \bmod 3$.
\end{remark}

\section{The exceptional temperaments $6$ and $12$}\label{sec:exceptional}

Theorem~\ref{thm:mod3} makes $3\mid n$ necessary.  It is not sufficient:
$B(6)=B(12)=0$.  These are finite facts with finite certificates, but the
\emph{shape} of the certificates locates the phenomenon: the standard
twelve-tone equal temperament is one of exactly two exceptions among all
multiples
of three.

\begin{proposition}[Certificates]\label{prop:certificates}
At $n=6$ there are exactly $4$ balanced choice functions with zigzag
chromatic, and at $n=12$ exactly $48$; in every case the closure sum
$\sum_k a(F(k)) \bmod n$ is nonzero.  Hence $B(6)=B(12)=0$.
\emph{(Computer-verified by two independent programs; at $n=12$ the $48$
closure sums realize every nonzero residue mod $12$---each of
$1,\dots,11$ occurs---and only $0$ is missed.)}
\end{proposition}

The histogram in Proposition~\ref{prop:certificates} is the right way to
see what kind of obstruction this is: \emph{not} a congruence (no residue
class is forbidden except the single value $0$) but a \emph{cancellation}.
Writing $b(n)$ for the number of balanced-and-closed zigzag choice
functions and $a(n)$ for balanced ones, equidistribution would predict
$b\approx a/n$; a character decomposition gives exactly
\[
b(n)\;=\;\frac1n\sum_{\psi}\ S(\psi),\qquad
S(\psi)=\sum_{F\ \mathrm{balanced,\ zigzag}}\psi\Bigl(\sum\nolimits_k a(F(k))\Bigr),
\]
over additive characters $\psi$ of $\Z_n$, and the exceptional
temperaments are precisely those where the nontrivial character sums
cancel the main term.  Extensive computation
(Section~\ref{sec:census}) shows the cancellation is a small-$n$
accident, not a law: the deficit changes sign with the parity of $n$
(for odd multiples of $3$ the value $0$ is the \emph{most} popular
closure sum; for even ones it is suppressed) and only at $n=6,12$ does the
suppression reach zero.  Every multiple of three from $15$ through at
least $36$ has balanced-closed zigzag choice functions in abundance
($16$ at $n=15$, rising to $216{,}436$ at $n=36$).

\begin{musical}
Why $12$?  Not for the reason zigzags fail in $10$TET or $14$TET---there
they are forbidden outright by the conservation law.  In $12$TET the
conservation law is satisfied, the harmonic plans exist ($48$ of them),
and each one misses closing the octave by some nonzero number of
semitones: every possible error except zero occurs.  Twelve-tone equal
temperament misses the zigzag by a kind of destructive interference,
a numerical coincidence shared only with the trivial $6$TET---an
exceptionalism that the histogram gives no advance sign of.  And as
Section~\ref{sec:drift} makes precise, what $12$TET withholds is only
closure: forty-four of those forty-eight harmonic plans are realized by
zigzag \emph{sequences} at their own nonzero drifts (the remaining four
are disconnected and realize nothing).
\end{musical}

\section{Existence: the pumping construction}\label{sec:pumping}

It remains to construct zigzag cycles for every multiple of three except
$6$ and $12$.  Cycles at $n=9,15,18$ are known by search; the content of
this section is a machine for turning a cycle at $n$ into a cycle at
$n+6$, which then settles all remaining $n$ by induction from the bases
$15$ and $18$.

\subsection{Level coordinates}\label{sec:levels}

Pair each class with its inversion partner: for $3 \le j$, the pair
$P_j=\{T_j,\ T_{n+1-j}\}$ at \emph{level} $j$.  Represent nodes by signed
representatives.  The crucial observation---the reason a
temperament-independent gadget can exist at all---is that in these
coordinates the six forms of \emph{both} classes at level $j$ become edges
on the symbolic nodes $\{\pm1,\ \pm(j{-}1),\ \pm j\}$ in a pattern that
depends on neither $j$ nor $n$; and each form's step value is a linear
function $\hat a = \mu\cdot j + c$ of the level, with $\mu\in\{-1,0,1\}$.
(Appendix~\ref{app:tables} tabulates the twelve patterns.)  Consequently:
processing levels in increasing order, the balance conditions close two
nodes per level, and the whole system is a walk in a \emph{fixed} transfer
automaton whose state records the imbalance carried on the current
frontier $\{\pm j\}$, with side-ledgers for the semitone nodes $\pm1$ and
for the \emph{integer} sum $\Sigma=\sum \hat a$, which is linear in levels
and satisfies: $F$ is closed iff $n \mid \Sigma$.

\subsection{The gadget lemma}\label{sec:gadget}

\begin{lemma}[Pumping gadgets]\label{lem:gadget}
There exist data $(X,\varphi,w)$ for each parity class as follows, all
conditions being identities of the $n$-independent transfer tables,
verified by finite computation:
\begin{itemize}[leftmargin=2em]
\item \emph{Odd chain.}  The base cycle $F_{15}$ with forms
$(\form{II},\form{V},\form{II},\form{I},\form{III},\form{IV},\form{VI},
\form{V},\form{II},\form{I},\form{III},\form{IV})$ on classes
$T_2,\dots,T_{13}$ is valid and zigzag, enters its top with frontier state
$X=(0,-1)$, is capped by form $\varphi=\form{VI}$, and has integer sum
$\Sigma=0$.  The three-level word
\[
w_{\mathrm{odd}}=(\form{VI},\form{I})\,(\form{III},\form{II})\,(\form{IV},\form{V})
\]
maps $X$ to $X$ through valid closures, contributes $(0,0)$ to the
semitone ledgers, has level-coefficient sum $\sum\mu=0$, and constant term
$-3\mu_\varphi$; hence re-capping with $\varphi$ three levels higher
preserves $\Sigma=0$ exactly.
\item \emph{Even chain.}  Likewise $F_{18}$ (forms
$\form{II},\form{I},\form{V},\form{II},\form{I},\form{V},\form{II},
\form{III},\form{IV},\form{V},\form{II},\form{I},\form{V},\form{II},\form{I}$
on $T_2,\dots,T_{16}$) with $X=(1,0)$, merged cap
$(\form{III},\form{IV})$, invariant $\Sigma = 1\cdot n$, and word
word
\[w_{\mathrm{even}}=(\form{V},\form{IV})\,(\form{II},\form{III})\,(\form{I},\form{VI})\]
satisfying the corresponding identity, so that after pumping
$\Sigma=1\cdot(n+6)$: closed again.
\end{itemize}
\end{lemma}

\begin{lemma}[Connectivity is inherited]\label{lem:connectivity}
In each chain, the pumped graph is connected whenever its predecessor is.
\end{lemma}

\begin{proof}
Odd chain: the removed cap edge is $(-\Lambda\to-1)$; the first gadget
level contributes the edge $(-\Lambda\to-1)$ \emph{with the same
endpoints} (the low-$\form{VI}$ form at level $\Lambda+1$), so the new
graph contains a replacement for the deleted edge, and the gadget's
remaining edges form a connected subgraph attached at $-1$ and $+1$.
Even chain: the deleted merged cap comprised the only two edges at the
node $n/2$, which therefore vanishes cleanly; the old path
$-1\to n/2 \to +(\Lambda{-}1)$ reroutes through the gadget hub $-\Lambda$,
which the gadget joins directly to $-1$, $+1$, $+(\Lambda{-}1)$ and
$+(\Lambda{+}1)$.  Both arguments concern only the fixed relative edge
patterns, hence hold at every pump.
\end{proof}

\begin{theorem}[Classification]\label{thm:classification}
Zigzag trichord cycles exist in $n$TET if and only if $3\mid n$ and
$n\notin\{6,12\}$.
\end{theorem}

\begin{proof}
Necessity of $3\mid n$ is Theorem~\ref{thm:mod3}; exclusion of $6,12$ is
Proposition~\ref{prop:certificates}.  For sufficiency: $n=9$ by
Example~\ref{ex:9tet}; for $n\equiv 3\ (\mathrm{mod}\ 6)$, $n\ge 15$, induct
from $F_{15}$ by Lemma~\ref{lem:gadget} (balance is local and preserved;
the semitone ledgers are unchanged; closure holds since the integer
invariant $\Sigma=c\cdot n$ is maintained, so $\Sigma \equiv 0$ modulo the
new $n+6$; the chromatic form is untouched; connectivity by
Lemma~\ref{lem:connectivity}); for $n\equiv 0\ (\mathrm{mod}\ 6)$, $n\ge 18$,
likewise from $F_{18}$.  \emph{(As an independent check, the pumped
cycles were verified from scratch---balance, closure, connectivity,
zigzag chromatic---at $n=21,27,33,39$ and $n=24,30,36,42$.)}
\end{proof}

\begin{musical}
The gadget is a six-trichord module---three inversion-paired levels of
contour assignments---that can be spliced into a working harmonic plan to
lift it from $n$-tone to $(n{+}6)$-tone equal temperament, and spliced
again, indefinitely.  That such a module can exist at all reflects a
structural fact: expressed in the right coordinates, the
grammar of trichord contours is the \emph{same in every temperament}; only
the boundary conditions change.  The witnesses found by search wear this
on their sleeves---their contour assignments are visibly periodic, like
crystal growth---and the gadget is nothing but one unit cell of the
crystal, isolated and certified.
\end{musical}
\part{Why zigzags are rare: the quantitative story}

\section{The dynamics of charge}\label{sec:dynamics}

The classification says where zigzags live; this section says why, even
there, they are so few.  The answer upgrades Theorem~\ref{thm:mod3} from a
single identity to a \emph{dynamical conservation law}.

Abstract the transfer process of Section~\ref{sec:levels} to its frontier:
the state is the pair of imbalances carried on the frontier nodes, a
$25$-state transfer automaton (imbalances in $[-2,2]^2$), of which
thirteen states are reachable, with transitions given by
the $36$ contour combinations of a level.  The chromatic initializes the
walk: scalar forms, being self-loops at $\pm1$---\emph{the only self-loops
in the entire network, for any $n$}---inject nothing and start the walk at
the neutral state $(0,0)$; zigzag forms are genuine edges and start it at
a charged state $(\pm1,0)$ or $(0,\pm1)$.

\begin{theorem}[Sector decomposition]\label{thm:sectors}
The reachable state space of the frontier automaton decomposes into
disjoint components:
\[
\mathcal{N}=\{(0,0),\,(1,1),\,(-1,-1)\}
\qquad\text{and}\qquad
\mathcal{C}\ (\text{ten states}),
\]
with the scalar initializations confined to $\mathcal{N}$ and the zigzag
initializations to $\mathcal{C}$; \emph{no transition leads from
$\mathcal{C}$ into $\mathcal{N}$} (the only states with a transition into
$(0,0)$ are the neutral three).  The separating invariant is
$\delta_+-\delta_-\bmod 3$: identically $0$ on $\mathcal{N}$, never $0$ on
$\mathcal{C}$.  The Perron roots are
\[
\lambda_{\mathcal N}=\frac{7+\sqrt{33}}{2}=6.372281\ldots,
\qquad
\lambda_{\mathcal C}=3\ \text{(exactly)}.
\]
\emph{(Finite computation: the automaton has $25$ states, thirteen of
them reachable---the three of $\mathcal{N}$ and the ten of
$\mathcal{C}$---and the component analysis is an exhaustive check.)}
\end{theorem}

\noindent The neutral transfer matrix, on the ordered basis $(0,0)$,
$(1,1)$, $(-1,-1)$, is
\[
M_{\mathcal N}=\begin{pmatrix} 6 & 1 & 1\\ 1 & 1 & 0\\ 1 & 0 & 1\end{pmatrix},
\qquad
\det(xI-M_{\mathcal N})=x^3-8x^2+11x-4=(x-1)(x^2-7x+4),
\]
so $\lambda_{\mathcal N}=(7+\sqrt{33})/2$ exactly.  Since
$\lambda_{\mathcal N}^2=7\lambda_{\mathcal N}-4$,
$\lambda_{\mathcal N}^{3}=45\lambda_{\mathcal N}-28=(259+45\sqrt{33})/2
=258.7527\ldots$, the growth per six steps of $n$.  The charged component
has ten states---$(-2,0)$, $(-1,0)$, $(-1,1)$, $(0,-2)$, $(0,-1)$,
$(0,1)$, $(0,2)$, $(1,-1)$, $(1,0)$, $(2,0)$---with out-degrees
$1,4,2,1,4,4,1,2,4,1$ in that order, and Perron root exactly $3$.

A zigzag thus \emph{behaves like} a topological defect in the transfer
network.
The analogy is to a conserved charge, not to topology in the usual
sense: no space, homology class or winding number is invoked, only the
$\Z/3$ invariant of Theorem~\ref{thm:mod3}.  The charge created by the
chromatic's contour can never be locally annihilated---there is no path
back to neutrality---and must thread the entire cycle, cancelling only
arithmetically at closure, which is what Theorem~\ref{thm:mod3} governs.
And the defect is costly: the charged sector supports strictly less
freedom per level.

The Perron data match the censuses, though the two sectors stand
differently.  In the neutral sector the six-level growth factor is the
exact algebraic number $\lambda_{\mathcal N}^{3}=(259+45\sqrt{33})/2
=258.7527\ldots$, and the measured growth of balanced scalar
configurations agrees with it to every computed digit.
Balanced zigzag configurations grow by the factor $24$ per six steps,
against $\lambda_{\mathcal C}^{3}=27$; the deficit $24/27=8/9$ we
attribute to the cost of conditioning the semitone-node ledger to bridge
back to zero inside the charged sector:

\begin{conjecture}\label{conj:eightninths}
The number of balanced zigzag choice functions satisfies
$a_{\mathrm{zig}}(n+6)/a_{\mathrm{zig}}(n)\to 24$, the factor $8/9$
relative to $\lambda_{\mathcal C}^{3}$ being the exact asymptotic price of
the semitone-ledger bridge in the charged sector.
\end{conjecture}

\begin{conjecture}\label{conj:recursion}
With $t=n/3$, the number $a(t)$ of balanced zigzag choice functions
satisfies $a(2j)=4\,a(2j-1)$ and $a(2j+1)=\frac{6(2j-1)}{2j}\,a(2j)$,
matching all computed values $t=2,\dots,12$:
$4,12,48,216,864,4320,17280,90720,362880,1959552,7838208$
(note $a(10)=9!$).
\end{conjecture}

\noindent Conjecture~\ref{conj:recursion} implies
Conjecture~\ref{conj:eightninths}: the recursion gives
$a(2j{+}2)/a(2j)=24(2j-1)/(2j)\to 24$, so the per-$\Delta n=6$ growth
tends to $24$ and the ratio to $\lambda_{\mathcal C}^{3}=27$ tends to
$8/9$.

\begin{musical}
Three nested reasons, then, why the zigzag is rare.  First, an arithmetic
admission ticket: two temperaments in three refuse it outright.  Second,
confinement: where admitted, the zigzag is a defect line running through
the whole cycle---once the chromatic breaks its run, the bookkeeping never
returns to neutral, and the number of ways to continue shrinks from about
$6.4$ per stage to exactly $3$.  Third, a toll on top: even within the
charged channel, closing the books costs a further constant factor.
Compounded, the effect is exponential: in $36$TET, harmonic plans with a
zigzag chromatic are outnumbered by scalar plans roughly $1.8$ million to
one: at $n=36$ there are $13{,}824{,}544{,}308{,}684$ balanced scalar
plans against $7{,}838{,}208$ balanced zigzag ones, a ratio of
$1{,}763{,}738$.  The zigzag is not locally difficult---it is globally expensive.
\end{musical}

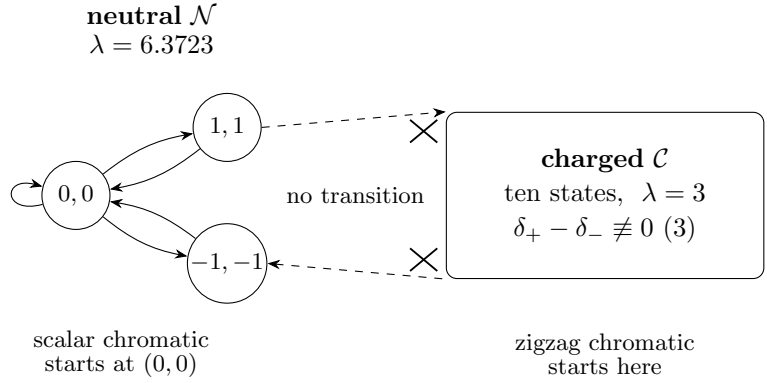
\begin{figure}[ht]
\centering
\begin{tikzpicture}[>={Stealth},x=1mm,y=1mm,font=\small,
  st/.style={draw,circle,inner sep=0.5pt,minimum size=9mm,font=\footnotesize}]
 \node[st] (n0) at (0,0)    {$0,0$};
 \node[st] (n1) at (20,9)   {$1,1$};
 \node[st] (n2) at (20,-9)  {$-1,-1$};
 \draw[->] (n0) to[loop left,looseness=6] (n0);
 \draw[->] (n0) to[bend left=14]  (n1);
 \draw[->] (n1) to[bend left=14]  (n0);
 \draw[->] (n0) to[bend right=14] (n2);
 \draw[->] (n2) to[bend right=14] (n0);
 \node[align=center] at (10,22)
   {\textbf{neutral} $\mathcal N$\\[-1pt] $\lambda=6.3723$};
 \node[align=center,font=\footnotesize] at (6,-21)
   {scalar chromatic\\[-2pt] starts at $(0,0)$};
 \node[draw,rounded corners,minimum width=42mm,minimum height=22mm,
   align=center] (C) at (70,0)
   {\textbf{charged} $\mathcal C$\\[1pt] ten states,\ \ $\lambda=3$\\[1pt]
    $\delta_+-\delta_-\not\equiv 0 \ (3)$};
 \node[align=center,font=\footnotesize] at (70,-21)
   {zigzag chromatic\\[-2pt] starts here};
 \draw[dashed,->] (n1.east) -- node[pos=0.5,above,font=\footnotesize]{}
   (C.north west);
 \draw[dashed,<-] (n2.east) -- (C.south west);
 \node[font=\footnotesize] at (37,0) {no transition};
 \draw[line width=0.7pt] (44.2,7.0) -- (47.8,10.0);
 \draw[line width=0.7pt] (44.2,10.0) -- (47.8,7.0);
 \draw[line width=0.7pt] (44.2,-10.0) -- (47.8,-7.0);
 \draw[line width=0.7pt] (44.2,-7.0) -- (47.8,-10.0);
\end{tikzpicture}
\caption{Schematic of the sector decomposition
(Theorem~\ref{thm:sectors}): the frontier dynamics never passes between
the neutral and charged components in either direction; in particular
charge, once created by a zigzag chromatic, can never be neutralized.}
\label{fig:sectors}
\end{figure}

\section{Crystals, necklaces, and the image--fiber decomposition}
\label{sec:crystal}

The map ```cycle $\mapsto$ its choice function''' abstracts a cycle to its
set of ordered trichords, forgetting temporal order.  By
Theorem~\ref{thm:reduction} the count of cycles factors through this map:
\[
\#\{\text{cycles}\}\;=\;\sum_{F\in \mathrm{Image}} \underbrace{m\cdot\ec(G_F)}_{\text{fiber of }F},
\]
and the rarity of zigzag cycles decomposes into rarity of the
\emph{image} (few zigzag harmonic plans) versus smallness of the
\emph{fibers} (few melodic realizations per plan).  Exact censuses on both
sides give:
\[
\frac{A(n)}{B(n)}
=\underbrace{\frac{\#\text{valid scalar }F}{\#\text{valid zigzag }F}}_{\text{image ratio}}
\times
\underbrace{\frac{\text{avg.\ scalar fiber}}{\text{avg.\ zigzag fiber}}}_{\text{fiber ratio}}:
\qquad
\begin{array}{c|ccc}
n & 15 & 18 & 21\\\hline
A/B & 18 & 4410 & 798.7\\
\text{image} & 7 & 301.8 & 59.6\\
\text{fiber} & 2.6 & 14.6 & 13.4
\end{array}
\]
The image ratio carries the exponential growth (it is the Perron gap of
Theorem~\ref{thm:sectors}, with a parity oscillation); the fiber ratio is
a bounded-looking degree effect---scalar plans keep busier semitone
nodes, and Eulerian counts (by the BEST theorem
\cite{AEdB1951,deBruijn1946}; for existence see Hierholzer
\cite{Hierholzer1873}) pay in factorials of those degrees.  The
rarity of zigzag cycles is thus \emph{primarily harmonic}, not melodic:
it is the set of admissible plans that is thin.

There is a second, non-temporal cyclic structure here.  In the natural
enumeration $012, 013, \dots, 01(n{-}2)$, consecutive classes are bonded
by a shared interval: $T_k$ and $T_{k+1}$ share interval class $k$.
Closing the loop, the last class
$T_{n-2}$ is bonded back to $T_2$ through the whole tone (since
$n-2\equiv-2$).  \emph{The enumeration of trichord types is itself a
closed necklace}, bonded by shared intervals; a cycle's harmonic plan is a
decoration of this necklace by contours; the balance conditions couple
each necklace site to its neighbours and to its inversion partner, so
that the working lattice is the necklace folded at its midpoint: the
level ladder of Section~\ref{sec:levels}.  The visibly quasi-periodic structure of the
valid decorations (the alternating $(\form{III},\form{IV})$ and
$(\form{V},\form{II})$ motifs of the witnesses; the pumping gadget as a
literal unit cell) justifies the word \emph{crystal}; and in this language
the zigzag chromatic is a \emph{dislocation}: a defect whose strain
cannot heal locally (Theorem~\ref{thm:sectors}) and must propagate around
the necklace with total winding $\equiv 0 \bmod 3$
(Theorem~\ref{thm:mod3}).

\begin{musical}
Every trichord cycle is two cycles at once.  One is temporal---the melody,
the order in which the windows sound.  The other is timeless: the necklace
of the trichord types themselves, strung by their shared intervals, which
closes on itself through the whole tone joining $01(n{-}2)$ back to $012$.
The Missing Zigzag Puzzle was never about the first cycle.  It is a fact
about lawful decorations of the second---about harmony, not melody---and
that is why it could be heard long before it could be explained.
\end{musical}

\section{Census}\label{sec:census}

Table~\ref{tab:census} collects the exact censuses, including values far
beyond direct search, obtained through the reduction
(Theorem~\ref{thm:reduction}) by enumerating valid choice functions with
transfer-pruned search and summing Eulerian counts.

\begin{table}[ht]
\centering
\begin{tabular}{rrrl}
\toprule
$n$ & scalar $A(n)$ & zigzag $B(n)$ & notes\\
\midrule
6,\,7 & 0 & 0 & no cycles at all\\
8  & 30 & 0 & \\
9  & 0 & 24 & all cycles zigzag; unique orbit\\
10 & 84 & 0 & \\
11 & 0 & 0 & no cycles at all\\
12 & 1764 & 0 & exceptional zero (Prop.~\ref{prop:certificates})\\
13 & 320 & 0 & \\
14 & 24596 & 0 & \\
15 & 10368 & 576 & \\
16 & 440830 & 0 & \\
17 & 159656 & 0 & \\
18 & 7{,}937{,}910 & 1800 & \\
19,\,20 & $>0$ & 0 & scalar existence by witness\\
21 & 66{,}473{,}604 & 83{,}232 & \\
24 & --- & 952{,}560 & \\
27 & --- & 30{,}352{,}896 & \\
$\ge 30$ & --- & $>0$ & all $n\equiv 0,3\ (6)$, by Theorem~\ref{thm:classification}\\
\bottomrule
\end{tabular}
\caption{The census.  Labelled counts (starting pitch fixed).  Dashes:
count not computed; positivity in the last row is
Theorem~\ref{thm:classification}, not a search.  $B(n)=0$ for all
$3\nmid n$ by Theorem~\ref{thm:mod3}; the table's zeros at such $n$ were
additionally confirmed by search.  Entries for $8\le n\le 18$ were
recomputed by direct enumeration, those for $n=21,24,27$ from the BEST
theorem.}
\label{tab:census}
\end{table}

Two open regularities.  First, total existence: cycles exist for every
computed $n\ge 8$ except $n=11$, and fail for $n\le 7$; we conjecture
existence for all $n\ge 12$.  The emptiness of $11$TET---and the
vanishing of the \emph{scalar} sector at $n=9$, where connectivity, not
arithmetic, is the killer---suggest that the scalar side has its own
story, not treated here.  Second, the balanced-level counts obey exact-looking
laws (Conjectures~\ref{conj:eightninths}--\ref{conj:recursion}); through
the BEST theorem, a full closed-form census reduces to summing spanning-tree
counts over the polytope of valid plans, a well-posed problem we leave
open.

\part{Compositional epilogue}

\section{Composing with trichord cycles}\label{sec:composition}

We close where the project began.  Cycles of this kind underlie the
author's \emph{```\dots still plenty of good music\dots'''}
\cite{FeldmanScore},\footnote{The title alludes to a remark reported of
Schoenberg---spoken rather than written, and not traceable to his
published writings---that there is still plenty of good music to be
written in C major, made when tonal writing had come to seem exhausted.
By the time of the piece, strict twelve-tone writing had itself become
as unfashionable as C major was then; in both cases the old resource
admits renewal.} a guitar piece recorded on the
\emph{Leonardo Music Journal} CD series volume accompanying
\cite{Polansky1997}.  A few consequences of the mathematics
bear directly on compositional practice.

\emph{The smoothness dial.}  By \eqref{eq:2SK} and
Lemma~\ref{lem:adjacent}, the number $S$ of semitone steps parametrizes a
cycle's conjunctness, and the census stratifies accordingly: in $12$TET
the $1764$ cycles fall into strata with $S=2,\dots,5$ of sizes
$18,\,36,\,558,\,1152$---from a single long chain of leaps to
near-saturated chromatic motion.  A composer may choose the stratum as one
chooses a register.

\emph{The unique cadence.}  In any temperament without zigzags---$12$TET
included---every cycle contains exactly one direct chromatic run and no
other adjacent semitone motion (Lemma~\ref{lem:adjacent}).  The run is a
structural downbeat the music can spotlight or suppress: a provably
unique event, available to be composed with as such.

\emph{The forced zigzag of $9$TET.}  The one temperament whose unique
cycle is zigzag (Example~\ref{ex:9tet}) offers the mirror experience: a
chromatic that \emph{cannot} be stated as a run, only broken.  The
$15$TET example (Example~\ref{ex:15tet}) allows both in one temperament.

\emph{Two ears: the fourths transform.}  Serial practice is often heard
to prize the semitone---though characteristically voiced as a seventh or
ninth, which the present construction accommodates automatically, since
it lives in pitch classes and leaves register free.  A genuinely
different harmonic resource comes from multiplication.  For any unit $u$
of $\Z_n$, the map $x\mapsto ux$ carries the half-step family
bijectively onto the family of cycles exhausting the trichord types that
contain interval class $u$, transporting \emph{every} theorem of this
paper intact.  In $12$TET the units yield exactly two families---built
on the two generating interval classes, semitone and perfect
fourth---so precisely two compositional ears are served: the chromatic,
and the quartal--quintal ear of much twentieth-century harmony.  The
transported facts are worth hearing in their new voicing: quartal cycles
of nine pitch classes exhaust all nine fourth-bearing trichord types;
the image of the chromatic is the quartal trichord $027$; and the
classification, conjugated, says that in $12$TET the $027$ always
appears as a \emph{direct stack of fourths} (or of fifths, descending)
and can never be broken---with exactly one such pure stack per cycle,
the quartal cadence.  The example cycle of Section~\ref{sec:intro}
transforms to
\[
\textrm{C}\quad \textrm{F}\quad \textrm{B}\flat\quad \textrm{A}\flat\quad
\textrm{C}\sharp\quad \textrm{C}\quad \textrm{G}\quad \textrm{E}\quad
\textrm{B},
\]
its opening stack $\textrm{C}$--$\textrm{F}$--$\textrm{B}\flat$ the
image of the chromatic run $\textrm{C}$--$\textrm{C}\sharp$--$\textrm{D}$
(Figure~\ref{fig:fourths}).

\begin{figure}[ht]
\centering
\includegraphics[width=1.00\textwidth]{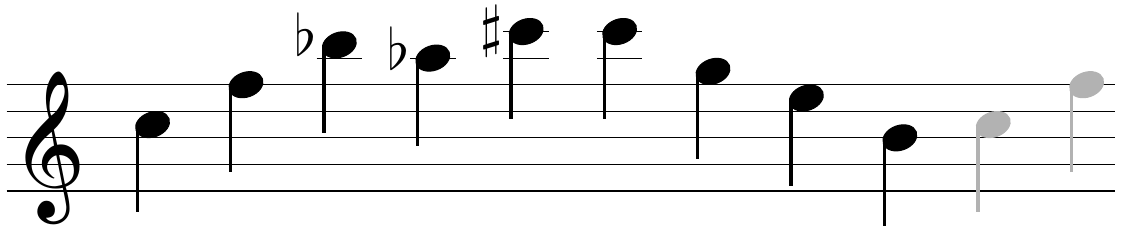}
\caption{The image of the cycle of Figure~\ref{fig:cycle12} under
$x\mapsto 5x$: a quartal cycle whose nine windows realize the nine
fourth-bearing trichord types.  The opening
C--F--B$\flat$ is the image of the chromatic run, and appears
necessarily as an unbroken stack of fourths.}
\label{fig:fourths}
\end{figure}
Note the division of labour: in Marsden's family
(Appendix~\ref{app:serial}) multiplication is a \emph{symmetry},
collapsing the census; here it is an \emph{intertwiner} between sibling
families, transporting the census.  The interval classes that are not
units ($2,3,4,6$ in $12$TET) generate families conjugate to nothing:
cycles exhausting the whole-tone-bearing trichords, say, fall outside
this paper's theory entirely.

\emph{Harmony versus melody, precisely.}  The reduction theorem licenses
a two-stage compositional workflow: first choose the harmonic plan (a
valid choice function---the contour of every trichord type at once), then
choose among its Eulerian realizations; the fiber counts of
Section~\ref{sec:crystal} measure exactly how much melodic freedom each
plan affords.  Plans rich in semitone hinges are melodically pliant;
zigzag plans, confined to the charged sector, are stiff---which a composer
may of course want.

\section{Open problems}\label{sec:open}

\begin{problem} Prove Conjectures~\ref{conj:eightninths} and
\ref{conj:recursion}; more ambitiously, give a closed-form census
$A(n),B(n)$ via character sums over the transfer product and the BEST
theorem.\end{problem}
\begin{problem} Classify the temperaments with no cycles at all
(known: $6,7,11$; conjecturally no others) and those with no
\emph{scalar} cycles (known: $9$; is it alone?).\end{problem}
\begin{problem}\label{prob:marsden} Find a conservation law governing the
presentation of the chromatic trichord in Marsden's all-trichord rings
(Theorem~\ref{thm:missing-scalar}), in the manner of
Theorem~\ref{thm:mod3}.\end{problem}
\begin{problem} Develop
the theory of presentation---how a distinguished member of a family is
realized in a universal cycle---for other u-cycle families (e.g.,
$k$-subsets \cite{Jackson1993,Hurlbert1994}); when is presentation
governed by a conservation law?\end{problem}
\begin{problem} Develop the theory for families built on
non-generating interval classes (e.g., cycles exhausting the
whole-tone-bearing trichord types of $12$TET), which are conjugate to no
half-step family and inherit none of the present results.\end{problem}
\begin{problem} Extend the classification to windows of width $w>3$
(tetrachord cycles), where the ```chromatic''' $0123$ admits many
contours.\end{problem}

\section*{Acknowledgment}
The computations, the reduction to choice functions, the weight-function
argument, the finite certificates, the pumping construction, the sector
analysis, and the Lean~4 development were carried out with the assistance
of Claude (Anthropic).  Responsibility for the mathematics and for any
errors is the author's alone.  Code sufficient to reproduce every
computation accompanies the paper, together with the Lean~4 development
of Appendix~\ref{app:lean}.

\section*{Declaration of generative AI use}
Generative artificial intelligence was used extensively in this work.
The computations, the reduction to choice functions, the
weight-function argument of Theorem~\ref{thm:mod3}, the finite
certificates for $n=6$ and $n=12$, the explicit construction of
Section~\ref{sec:pumping}, the sector analysis of
Section~\ref{sec:dynamics}, and the Lean~4 development of
Appendix~\ref{app:lean} were carried out in collaboration with Claude
(Anthropic), which also assisted in drafting the manuscript.  All
mathematical content has been reviewed by the author; in addition, the
principal results are machine verified (Appendix~\ref{app:lean}), and
every enumerative claim in the paper was recomputed by two independent
programs using different representations.  Responsibility for the
content rests with the author.

\section*{Disclosure statement}
The author reports there are no competing interests to declare.

\section*{Funding}
No funding was received for this work.

\section*{Data availability statement}
\ifanon
\else
The Lean~4 development and the programs reproducing every enumerative
claim are openly available at
\texttt{https://doi.org/10.5281/zenodo.22040738} and developed at
\texttt{https://github.com/DavidVFeldman/missing-zigzag-new}.
\fi

\appendix

\section{Form and transfer tables}\label{app:tables}
The six forms of $T_k$ as step pairs appear in \eqref{eq:forms}.  In level
coordinates (level $j$; nodes $\pm1$, inner $\pm(j{-}1)$, outer $\pm j$;
$\hat a=\mu j+c$):

\begin{center}
\begin{tabular}{lcc@{\qquad}lcc}
\toprule
\multicolumn{3}{c}{low class $T_j$} & \multicolumn{3}{c}{high class $T_{n+1-j}$}\\
form & edge & $(\mu,c)$ & form & edge & $(\mu,c)$\\
\midrule
\form{I}  & $+1\to +(j{-}1)$ & $(0,1)$  & \form{I}  & $+1\to -j$        & $(0,1)$\\
\form{II} & $+j\to -(j{-}1)$ & $(1,0)$  & \form{II} & $-(j{-}1)\to +j$  & $(-1,1)$\\
\form{III}& $-1\to +j$       & $(0,-1)$ & \form{III}& $-1\to -(j{-}1)$  & $(0,-1)$\\
\form{IV} & $+(j{-}1)\to -j$ & $(1,-1)$ & \form{IV} & $-j\to +(j{-}1)$  & $(-1,0)$\\
\form{V}  & $-j\to +1$       & $(-1,0)$ & \form{V}  & $+(j{-}1)\to +1$  & $(1,-1)$\\
\form{VI} & $-(j{-}1)\to -1$ & $(-1,1)$ & \form{VI} & $+j\to -1$        & $(1,0)$\\
\bottomrule
\end{tabular}
\end{center}
These patterns are independent of $j$ and $n$; all transfer computations,
the gadget verification, and Theorem~\ref{thm:sectors} are finite
consequences.

\section{Gadget and base data}\label{app:gadget}
The base plans $F_{15}$, $F_{18}$ and gadget words
$w_{\mathrm{odd}}, w_{\mathrm{even}}$ appear in Lemma~\ref{lem:gadget}.
Pumped plans were re-verified from scratch at
$n=21,27,33,39$ (odd chain) and $n=24,30,36,42$ (even chain).  For
example, the odd chain's first pump yields a valid zigzag plan at $21$TET
whose Eulerian realizations number $83{,}232/ (18\cdot\text{plans})$ on
average per plan; complete listings are in the ancillary files.

\section{Certificate summary for $n=6,12$}\label{app:cert}
$n=6$: the $4$ balanced zigzag plans have closure sums $\{2,4\}$ (each
twice); none is $0$.  $n=12$: the $48$ balanced zigzag plans have closure
sums distributed
\[1{:}4,\quad 2{:}6,\quad 3{:}4,\quad 4{:}4,\quad 5{:}4,\quad 6{:}4,\quad
7{:}4,\quad 8{:}4,\quad 9{:}4,\quad 10{:}6,\quad 11{:}4,\]
every nonzero residue, never $0$.
Complete listings (each plan as an $(n-3)$-tuple of forms: $3$-tuples at
$n=6$, $9$-tuples at $n=12$) are in the
ancillary files; two independently written programs agree.

\section{Serial ancestry, and cycles without repeated pitch classes}
\label{app:serial}

Everything in this appendix is established computationally rather than
formally verified, and none of it is used in the proof of
Theorem~\ref{thm:classification}.

\subsection{A serial family tree}

The objects of this paper descend recognizably from twelve-tone
serialism, and it is worth placing them in that lineage.  A Schoenbergian
row states each pitch class exactly once and imposes no condition on its
windows.  The classic refinements go two ways.  Berg's \emph{Lyric Suite}
row is \emph{all-interval}: its eleven consecutive dyads exhaust the
eleven ordered intervals---an exhaustiveness condition on windows of
width~$2$, on top of pitch-class injectivity; the enumeration of such
rows was an early landmark of computational music theory
\cite{BMF1965}, corrected by Cohen \cite{Cohen1972} and completed by
Morris and Starr \cite{MorrisStarr1974}, whose census of $3856$ rows in
normal form ($1928$ up to inversion) remains standard; the corpus has
since been revisited with network methods \cite{Nardelli2020}.  Webern's Concerto op.~24 derives its row from a
\emph{single} trichord type stated four times---exhaustive repetition of
one window class. Carter's all-trichord hexachord $\{012478\}$ contains a
representative of every trichord class \cite{Schiff1998,Whittall2008}.

The trichord cycles of this paper occupy a natural cell of this table:
exhaustiveness of \emph{width-$3$ windows} over the semitone-bearing
types, with pitch classes left \emph{free} to repeat.  The final
serial discipline is then restored by force:

\begin{definition}
An \emph{all-trichord near-row} is a trichord cycle
(Definition~\ref{def:cycle}) whose $n-3$ pitch classes are distinct---%
thus an arrangement of all but three of the $n$ pitch classes, whose
windows exhaust the semitone trichord types.
\end{definition}

Near-rows are the intersection of two requirements of quite different
combinatorial character, and the tight structure of the main text does
not survive the intersection.  The next subsection says precisely why.

\subsection{All-trichord rings, and the missing scalar}

There is one existing object that sits closer still.  By a numerical
coincidence special to
$n=12$, the number of trichord classes up to transposition \emph{and}
inversion is exactly twelve---the length of a row---so a twelve-tone row
read cyclically has exactly enough windows to exhaust them.  Marsden \cite{Marsden2012}, reporting in passing the first musical
computation he undertook, describes exactly these objects---twelve-note
series which, read circularly, realize all twelve $T/I$ trichord classes
among their twelve consecutive triples---states that there are only
four, and lists them.  We call them \emph{all-trichord rings}.  (The linear variant, after
Babbitt and Morris, drops the wrap and necessarily omits two classes,
conventionally $036$ and $048$ \cite{Morris1987}.)  His count is confirmed here independently, and so is his list: an
exhaustive search finds $192$ labelled rings falling into four orbits
under transposition, inversion, retrograde and rotation, and the four
series printed in \cite{Marsden2012} are precisely orbit
representatives of those four.  But the
all-trichord property enjoys one symmetry that the present family lacks: the
cycle-of-fourths transform $M_5\colon x\mapsto 5x$ permutes the twelve
$T/I$ trichord classes and therefore maps rings to rings, and it pairs
Marsden's four ($1\leftrightarrow2$, $3\leftrightarrow4$).  Under the
full group of twelve-tone operations including $M$ there are thus
\emph{essentially only two} all-trichord rings.  (No such collapse is
available in the main text: $M_5$ destroys the semitone-bearing property,
so the two families differ even in their symmetry groups.)

Our family and Marsden's differ in every convention---transposition-only
versus $T/I$ classes, semitone-bearing types versus all types, free
versus distinct pitch classes, all $n$ versus $n=12$ alone---which makes
the following comparison possible:

\begin{theorem}\label{thm:missing-scalar}
In every all-trichord ring, the chromatic trichord is presented in
zigzag; no ring presents it as a chromatic run.  \emph{(Finite
verification: all $192$ labelled rings, and in particular each of the
four series printed in \cite{Marsden2012}.)}
\end{theorem}

\begin{musical}
Twelve-tone equal temperament thus houses two exhaustive trichordal
cycles that are exact chromatic opposites.  In the cycles of this
paper, the chromatic trichord \emph{must} be a run and cannot be broken;
in Marsden's rings it \emph{must} be broken and cannot run.  (Compare
also $9$TET, the one temperament in the present family where the scalar, not
the zigzag, goes missing.)  Theorem~\ref{thm:missing-scalar} rests on exhaustive enumeration; a
conservation law governing Marsden's family in the manner of
Theorem~\ref{thm:mod3} is not known, and finding one is
Problem~\ref{prob:marsden}.
\end{musical}

\subsection{Injectivity is self-avoidance, and lives in the fiber}

Every exact result in this paper runs through the Reduction
(Theorem~\ref{thm:reduction}), whose whole power comes from quotienting
by transposition: forms, choice functions, balance, the conservation law,
the transfer automaton---all are statements about melodic \emph{steps},
blind to absolute pitch.  Injectivity is precisely the information this
quotient destroys: a cycle has distinct pitch classes if and only if the
\emph{integrated} walk $c_i=\sum_{j<i}d_j$ is injective---a
self-avoidance condition on the partial sums of the Eulerian circuit.
Distinctness therefore lives entirely in the fiber of the abstraction
map of Section~\ref{sec:crystal}, invisible to the harmonic plan; and
self-avoidance constraints are the classic generators of non-exact
combinatorics.  This is the structural reason to expect censuses but not
closed forms below---while noting the compensation that everything
remains finitely and efficiently computable.

\subsection{The independence heuristic}

Are the two requirements asymptotically independent?  A naive model
treats the $m-1=n-4$ free pitch classes of a labelled cycle as uniform,
giving
\[
p(n)\;=\;\Pr[\text{all distinct}]
\;=\;\frac{(n-1)!}{6\,n^{\,n-4}}
\;\sim\;\frac{\sqrt{2\pi n}\;n^{3}}{6}\,e^{-n},
\]
by Stirling's formula applied to $(n-1)!$.  Only the factor $e^{-n}$
matters below; the polynomial prefactor plays no part in the
exponential-scale comparison.
The censuses of Section~\ref{sec:census} are consistent with
exponential growth $A(n)+B(n)\asymp\lambda^{n}$; the observed per-step
ratios oscillate with the parity of $n$ and settle near $4$ for
$n\ge 14$, so that $\lambda>e$ (no asymptotic is proved here).  The
expected number of near-rows then grows like $(\lambda/e)^{n}$: the
heuristic predicts not mere nonemptiness but exponential abundance.  The data agree, and quantify the
correction: the ratio of the observed distinct fraction to $p(n)$ grows
only sub-exponentially (Table~\ref{tab:iid}), consistent with
independence at the exponential scale; the excess over $1$ is explained
qualitatively by the window structure itself, which already forbids
repeats at distances $1$ and $2$ and thus removes the dominant collision
terms of the naive model.

\begin{table}[ht]
\centering
\begin{tabular}{rccc}
\toprule
$n$ & observed fraction & iid $p(n)$ & ratio\\
\midrule
12 & 0.2347 & $1.55\times10^{-2}$ & 15.2\\
13 & 0.1250 & $7.53\times10^{-3}$ & 16.6\\
14 & 0.0886 & $3.59\times10^{-3}$ & 24.7\\
15 & 0.0263 & $1.68\times10^{-3}$ & 15.7\\
16 & 0.0183 & $7.74\times10^{-4}$ & 23.7\\
17 & 0.0175 & $3.52\times10^{-4}$ & 49.8\\
18 & 0.0054 & $1.58\times10^{-4}$ & 34.4\\
21 & 0.0010 & $1.35\times10^{-5}$ & 76.9\\
\bottomrule
\end{tabular}
\caption{Fraction of labelled cycles with distinct pitch classes, against
the independence model.  The correction ratio grows sub-exponentially:
the data are consistent with independence at the exponential scale.}
\label{tab:iid}
\end{table}

The delicate case is the zigzag sector, whose growth constant per unit
$n$ is $\approx 24^{1/6}\cdot(\text{fiber growth})\approx 2.7$---%
numerically indistinguishable from $e$.  Zigzag near-rows thus sit on the
heuristic's \emph{critical line}: their expected number is
$(\lambda_B/e)^n\times(\text{polynomial})$ with $\lambda_B/e\approx 1$,
so existence is decided by the corrections, not the exponent.  The data
(Table~\ref{tab:nearrows}) show the polynomial factors winning---so far,
with one sporadic failure.

\subsection{Census of near-rows}

\begin{table}[ht]
\centering
\begin{tabular}{rrrrr}
\toprule
$n$ & scalar $A_d(n)$ & zigzag $B_d(n)$ & orbits (sc./zz.) & fraction of all cycles\\
\midrule
8  & 30 & 0 & 2 / 0 & 1.000\\
9  & 0 & 24 & 0 / 1 & 1.000\\
10 & 84 & 0 & 4 / 0 & 1.000\\
11 & 0 & 0 & --- & ---\\
12 & 414 & 0 & 13 / 0 & 0.235\\
13 & 40 & 0 & 1 / 0 & 0.125\\
14 & 2178 & 0 & 50 / 0 & 0.089\\
15 & 240 & 48 & 5 / 1 & 0.026\\
16 & 8086 & 0 & 159 / 0 & 0.018\\
17 & 2800 & 0 & 50 / 0 & 0.018\\
18 & 43230 & 0 & 723 / 0 & 0.005\\
19 & 11264 & 0 & 176 / 0 & ---\\
20 & 242522 & 0 & 3575 / 0 & ---\\
21 & 67464 & 1584 & 937 / 22 & 0.001\\
24 & --- & 2100 & --- & ---\\
\bottomrule
\end{tabular}
\caption{All-trichord near-rows: labelled counts by chromatic contour.
For $n\le 10$ every cycle is automatically a near-row.  The value
$B_d(24)=2100$ is $4m$ times the pinned count $25$, the four zigzag
contours being equinumerous by the retrograde--inversion symmetry (an
identity checked exactly against direct enumeration at $n=15,21$).}
\label{tab:nearrows}
\end{table}

Three remarks.  First, near-rows exist at every computed $n\ge 12$ and
their count grows briskly ($\approx 2.3$ per unit $n$), in line with the
heuristic.  Second, \emph{zigzag} near-rows exist at $n=9,15,21,24$ but
fail at $n=18$---even though $18$TET has $1800$ zigzag cycles, every one
repeats a pitch class.  After closure and weak connectivity are imposed,
only twenty zigzag harmonic plans remain at $n=18$---far fewer than the
$864$ merely balanced ones counted in Section~\ref{sec:dynamics}---and
every Eulerian realization of every one of them repeats a pitch class.
Twenty misses is within statistical reason on the critical line; whether the failure at $18$ is sporadic (as $\{6,12\}$
proved to be one level up) or structural is open.  Third, the unique
zigzag near-row orbit of $15$TET,
\[
(0,\,2,\,1,\,7,\,8,\,12,\,9,\,10,\,5,\,11,\,4,\,3),
\]
twelve distinct pitch classes of fifteen with a broken chromatic, may be
the most concentrated compositional object this study has produced.
\begin{figure}[ht]
\centering
\includegraphics[width=0.98\textwidth]{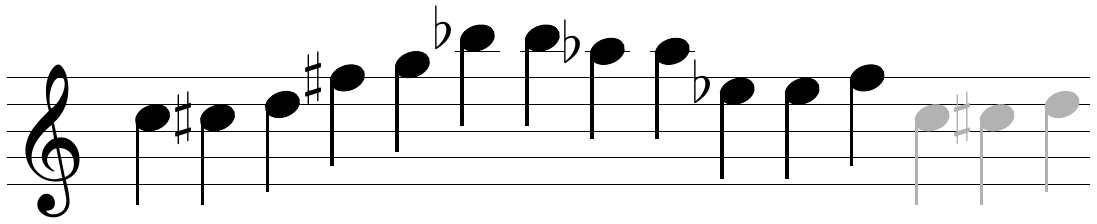}
\caption{A twelve-tone row arising from a drift-$3$ trichord spiral.  Its
ten consecutive trichords realize all nine semitone-bearing types, the
chromatic type twice.  The three grey notes begin the next row of the
chain, the $T_3$-transposition of this one, elided with it by a shared
trichord.}
\label{fig:row}
\end{figure}

\subsection{The omitted trichord}

A near-row in $12$TET omits three pitch classes.  The omitted set is an
arbitrary trichord---it is the \emph{complement} of the pitches used,
not a window, and so ranges over all trichord types, not merely the
semitone-bearing types that the windows must realize.  Which trichords
can be omitted?  The answer, over all $414$ near-rows, is sharply structured
(classes up to transposition; inversional pairs necessarily
equinumerous):

\begin{center}
\begin{tabular}{lr@{\qquad}lr}
\toprule
omitted type & count & omitted type & count\\
\midrule
$036$ (diminished) & 144 & $037/038$ (minor/major) & $18+18$\\
$012$ (chromatic) & 72 & $013/01\mathrm{T}$ & $18+18$\\
$016/017$ (Viennese) & $36+36$ & $014/019$ & $18+18$\\
$024$ (whole-tone) & 18 & &\\
\bottomrule
\end{tabular}
\end{center}

\noindent Eight transposition classes never occur as the omitted set:
the augmented triad $048$, and the types $015/018$, $025/029$, $026/028$,
$027$.  That the \emph{diminished triad} is by far the most dispensable
trichord, that the chromatic can be omitted from a structure whose whole
purpose is chromatic exhaustion, and that the \emph{augmented triad can
never be spared}, are facts we can verify but not yet explain; the last
seems the most tractable (the complement of an augmented triad is a union
of three chromatic trichords, a strong constraint on the available
steps).  The remaining entries of the table are unexplained.  The two set classes singled out here---diminished as the most
dispensable, augmented as indispensable---are precisely the two that the
Babbitt--Morris all-trichord row is defined by excluding, namely
SC 3-10 $[036]$ and SC 3-12 $[048]$ \cite{MorrisETT}.

\begin{musical}
For the composer: a near-row is a genuinely serial object---an ordering
of nine of the twelve pitch classes---that carries, in addition, the full
trichordal exhaustion of this paper, the unique chromatic cadence of
Lemma~\ref{lem:adjacent}, and a choice of which trichord to leave silent.
The table says the silence is most naturally diminished; that it can,
with more effort, be the chromatic itself (a cycle exhausting all
semitone trichord types while omitting three chromatically adjacent
pitches); and that it can never be augmented.
\end{musical}

\subsection{From drift back to the twelve-tone row}\label{app:rowsegue}

Drift (Section~\ref{sec:drift}) opens a passage from our cycles back
into the conventional twelve-tone world.  The first twelve notes of a
drift-$s$ spiral are its nine-note period followed by the
$s$-transposition of its opening window; if the period's pitches are
distinct, those twelve notes form an aggregate---a twelve-tone
row---exactly when
\[
\text{the omitted trichord}\;=\;T_s(\text{some window of the cycle}),
\]
welding the drift construction to the omitted-trichord analysis above.
When the condition holds, the spiral does not stop at one row: by
periodicity it is an unending \emph{chain of twelve-tone rows}, each the
$T_s$-transposition of the last, elided with its neighbour by a shared
trichord---row-chaining in the classical serial manner, closing after
$12/\gcd(s,12)$ rows.  Each such row has a property with no name we know
of: its ten consecutive trichords realize all nine semitone-bearing
trichord types, with exactly one type---the elision trichord---stated
twice.

Exhaustive search over all distinct-pitch drift spirals in $12$TET
($2358$ across the eleven nonzero drifts, of which $100$ are zigzag)
finds $152$ spiral--rotation pairs satisfying the aggregate condition,
yielding $34$ such rows up to transposition.  They arise at drifts $2$
through $10$ but never at drift $\pm 1$; the richest drifts are $\pm3$.
A specimen at drift $3$:
\[
\textrm{C}\ \textrm{C}\sharp\ \textrm{D}\ \textrm{F}\sharp\ \textrm{G}\
\textrm{B}\flat\ \textrm{B}\ \textrm{A}\flat\ \textrm{A}\ \textrm{E}\flat\
\textrm{E}\ \textrm{F},
\]
whose chain descends (or ascends) by minor thirds, closing after four
rows.  And one more finite fact, verified over the complete enumeration:

\begin{theorem}\label{thm:rowgate}
Every twelve-tone row arising as an aggregate of a trichord spiral in
$12$TET is scalar: none of the $100$ zigzag spirals extends to a row.
\end{theorem}

\begin{musical}
This is the missing zigzag's third refusal, and the most telling.  The
broken chromatic is not absent from twelve-tone equal temperament's
interval grammar---Section~\ref{sec:drift} produced it freely---but it
holds itself apart from the pitch world at every gate: it cannot close
into a nine-note cycle (drift zero is unattainable), and now we find
its spirals cannot align with the total chromatic either.  In $12$TET
the zigzag chromatic and the aggregate are, in every sense this paper
can measure, incompatible.  A composer who wants both has the
classification's advice: change the temperament.
\end{musical}

\subsection{Twelve-tone tables}\label{app:tables12}

For the reader who wants material to hand, the two twelve-tone
enumerations of this appendix in full.  Pitch classes are written
$0,\dots,9,\mathrm{T},\mathrm{E}$.

Table~\ref{tab:nearrows12} lists the all-trichord near-rows of $12$TET:
the $414$ labelled examples fall into $13$ classes under transposition,
inversion, retrograde and rotation, and one representative of each is
given, normalized to begin $0\,1\,2$ where possible.  Each uses nine of
the twelve pitch classes, its nine windows realize the nine
semitone-bearing trichord types once each, and---as
Theorem~\ref{thm:classification} requires---the chromatic window is in
every case a direct run.  The third column names the trichord type formed by the three omitted
pitch classes.  The fourth marks the three near-rows invariant under
retrograde-inversion: their stabilizer in the symmetry group has order
two, so each accounts for $18$ of the $414$ labelled examples where the
other ten account for $36$, and $10\cdot 36 + 3\cdot 18 = 414$.  For a
composer this is the familiar consequence that an RI-symmetric row has
$24$ rather than $48$ distinct forms.

\begin{table}[ht]
\centering\small
\begin{tabular}{rlll}
\toprule
 & near-row & omitted & sym. \\
\midrule
1 & \texttt{0\,1\,2\,4\,5\,9\,T\,6\,7} & $037$ &  \\
2 & \texttt{0\,1\,2\,4\,5\,T\,E\,7\,8} & $036$ &  \\
3 & \texttt{0\,1\,2\,5\,4\,E\,3\,T\,9} & $012$ & RI \\
4 & \texttt{0\,1\,2\,5\,4\,E\,T\,7\,8} & $036$ &  \\
5 & \texttt{0\,1\,2\,5\,6\,E\,T\,8\,7} & $016$ &  \\
6 & \texttt{0\,1\,2\,6\,3\,4\,T\,E\,8} & $024$ & RI \\
7 & \texttt{0\,1\,2\,6\,7\,9\,T\,3\,4} & $036$ &  \\
8 & \texttt{0\,1\,2\,6\,7\,T\,9\,4\,3} & $036$ &  \\
9 & \texttt{0\,1\,2\,7\,6\,9\,8\,3\,4} & $017$ &  \\
10 & \texttt{0\,1\,2\,7\,8\,T\,E\,3\,4} & $014$ &  \\
11 & \texttt{0\,1\,2\,8\,9\,6\,T\,E\,3} & $013$ &  \\
12 & \texttt{0\,1\,2\,9\,8\,E\,T\,3\,4} & $012$ &  \\
13 & \texttt{0\,1\,2\,E\,T\,5\,9\,4\,3} & $012$ & RI \\
\bottomrule
\end{tabular}
\caption{The $13$ all-trichord near-rows of $12$TET up to transposition,
inversion, retrograde and rotation.}
\label{tab:nearrows12}
\end{table}

Table~\ref{tab:rows12} lists the twelve-tone rows obtained from
drift spirals (Section~\ref{app:rowsegue}): a nine-note distinct-pitch
drift-$s$ cycle followed by the $s$-transposition of its first window.
There are $34$ up to transposition, each normalized to begin on $0$; the
second column gives the drift $s$, so that the row chains with its own
$T_s$-transposition and closes after $12/\gcd(s,12)$ rows.  Every one of
the $34$ is scalar (Theorem~\ref{thm:rowgate}), and in each the ten
consecutive trichords realize all nine semitone-bearing types with
exactly one repeated.

\begin{table}[ht]
\centering\small
\begin{tabular}{rll@{\qquad}rll}
\toprule
 & row & $s$ &  & row & $s$ \\
\midrule
1 & \texttt{0\,1\,6\,7\,4\,5\,9\,T\,E\,2\,3\,8} & $2$ & 18 & \texttt{0\,E\,T\,1\,2\,7\,3\,8\,9\,6\,5\,4} & $6$ \\
2 & \texttt{0\,5\,6\,9\,T\,E\,3\,4\,1\,2\,7\,8} & $2$ & 19 & \texttt{0\,E\,T\,7\,8\,1\,9\,2\,3\,6\,5\,4} & $6$ \\
3 & \texttt{0\,1\,2\,6\,7\,T\,E\,8\,9\,3\,4\,5} & $3$ & 20 & \texttt{0\,1\,9\,T\,E\,6\,5\,3\,2\,7\,8\,4} & $7$ \\
4 & \texttt{0\,1\,2\,8\,9\,6\,7\,T\,E\,3\,4\,5} & $3$ & 21 & \texttt{0\,8\,9\,2\,1\,E\,T\,5\,6\,7\,3\,4} & $7$ \\
5 & \texttt{0\,5\,1\,6\,7\,9\,T\,E\,2\,3\,8\,4} & $3$ & 22 & \texttt{0\,E\,T\,1\,2\,9\,8\,3\,4\,7\,6\,5} & $7$ \\
6 & \texttt{0\,8\,1\,2\,5\,6\,7\,9\,T\,3\,E\,4} & $3$ & 23 & \texttt{0\,E\,T\,1\,9\,2\,3\,8\,4\,7\,6\,5} & $7$ \\
7 & \texttt{0\,E\,T\,5\,4\,8\,9\,6\,7\,3\,2\,1} & $3$ & 24 & \texttt{0\,E\,T\,1\,2\,5\,4\,9\,3\,8\,7\,6} & $8$ \\
8 & \texttt{0\,E\,T\,5\,9\,6\,7\,4\,8\,3\,2\,1} & $3$ & 25 & \texttt{0\,E\,T\,3\,9\,2\,1\,4\,5\,8\,7\,6} & $8$ \\
9 & \texttt{0\,E\,T\,6\,7\,4\,5\,9\,8\,3\,2\,1} & $3$ & 26 & \texttt{0\,1\,2\,6\,5\,8\,7\,3\,4\,9\,T\,E} & $9$ \\
10 & \texttt{0\,1\,2\,9\,3\,T\,E\,8\,7\,4\,5\,6} & $4$ & 27 & \texttt{0\,1\,2\,7\,3\,6\,5\,8\,4\,9\,T\,E} & $9$ \\
11 & \texttt{0\,1\,2\,E\,T\,7\,8\,3\,9\,4\,5\,6} & $4$ & 28 & \texttt{0\,1\,2\,7\,8\,4\,3\,6\,5\,9\,T\,E} & $9$ \\
12 & \texttt{0\,1\,2\,E\,3\,T\,9\,4\,8\,5\,6\,7} & $5$ & 29 & \texttt{0\,4\,E\,T\,7\,6\,5\,3\,2\,9\,1\,8} & $9$ \\
13 & \texttt{0\,1\,2\,E\,T\,3\,4\,9\,8\,5\,6\,7} & $5$ & 30 & \texttt{0\,7\,E\,6\,5\,3\,2\,1\,T\,9\,4\,8} & $9$ \\
14 & \texttt{0\,4\,3\,T\,E\,1\,2\,7\,6\,5\,9\,8} & $5$ & 31 & \texttt{0\,E\,T\,4\,3\,6\,5\,2\,1\,9\,8\,7} & $9$ \\
15 & \texttt{0\,E\,3\,2\,1\,6\,7\,9\,T\,5\,4\,8} & $5$ & 32 & \texttt{0\,E\,T\,6\,5\,2\,1\,4\,3\,9\,8\,7} & $9$ \\
16 & \texttt{0\,1\,2\,5\,4\,E\,3\,T\,9\,6\,7\,8} & $6$ & 33 & \texttt{0\,7\,6\,3\,2\,1\,9\,8\,E\,T\,5\,4} & $10$ \\
17 & \texttt{0\,1\,2\,E\,T\,5\,9\,4\,3\,6\,7\,8} & $6$ & 34 & \texttt{0\,E\,6\,5\,8\,7\,3\,2\,1\,T\,9\,4} & $10$ \\
\bottomrule
\end{tabular}
\caption{The $34$ twelve-tone rows arising from trichord spirals in
$12$TET, up to transposition, with the drift $s$ of the generating
spiral.}
\label{tab:rows12}
\end{table}

\section{Machine verification}\label{app:lean}

The principal results of Part~II have been formalized in Lean~4
(toolchain \texttt{v4.28.0}, Mathlib \texttt{8f9d9cf}) and machine
checked.  The development accompanies this paper; the following is a
summary of exactly what is and is not verified.

\emph{Formal setting.}  The primary formal object is the step word, not
the pitch cycle: a map $D:\Z_{n-3}\to\Z_n$ whose consecutive pairs
$(D_i,D_{i+1})$ realize each class $T_k$, $2\le k\le n-2$, exactly once
(\texttt{IsIntervalCycle}).  Pitch cycles are the differenced special
case (\texttt{IsCycle}), and drift is defined for the general object.
This is not a convenience: in this setting the conservation law is a
\emph{telescoping identity}, and the formal proof of
Theorem~\ref{thm:mod3} uses no graph theory, no Eulerian circuits, and
no multiset balance.  It is, accordingly, formalized in the strong form
of Section~\ref{sec:drift}: any interval cycle of \emph{any} drift.

\emph{Verified.}  With no axioms beyond Lean's three standard ones
(\texttt{propext}, \texttt{Classical.choice} and \texttt{Quot.sound}):
the six-form characterization of window realizations; the chromatic
cadence lemma (Lemma~\ref{lem:adjacent}) in full, corollary included---%
that two consecutive semitone steps agree in direction and span the
chromatic window presented scalarly, that a cycle has at most one such
adjacent pair, and that a zigzag cycle has none; the weight
computations and
the fee-sum identity; the fact that a step pair realizes at most one
class, and that every window of a cycle realizes one (a pigeonhole
over the $n-3$ classes and the $n-3$ positions); \textbf{the
conservation law} (Theorem~\ref{thm:mod3}), in both the interval-cycle
and the pitch-cycle form; the emptiness of $6$TET; the explicit zigzag witnesses at
$n=9,15,18$; and \textbf{the existence half of the classification}---%
zigzag cycles for \emph{every} $n\ge 8$ with $3\mid n$ and $n\ne 12$,
not merely for the computed instances.  Also verified by kernel
evaluation: that $7$TET and $11$TET admit no trichord cycle
(Table~\ref{tab:census}), and that $9$TET admits no scalar
cycle---indeed that no $9$TET cycle has any window whose two steps are
equal semitones, which is stronger.

Additionally verified, with the two compiler-reduction axioms
(\texttt{Lean.ofReduceBool}, \texttt{Lean.trustCompiler}) arising from
one reflected finite computation: \textbf{the twelve-tone
exclusion}, that no trichord cycle in $12$TET presents the chromatic in
zigzag.
The certificate is the $6^9$ enumeration of form-choice plans of
Proposition~\ref{prop:certificates}, reduced to a Boolean check proved
equivalent to the mathematical condition, and evaluated by compiled
reflection.  What it establishes is the conclusion drawn there for
$n=12$---that no balanced, closed plan has a zigzag chromatic---and not
the tallies $4$ and $48$, which rest on the two independent enumerations
reported in Section~\ref{sec:exceptional}.

Consequently \textbf{Theorem~\ref{thm:classification} is machine
verified in both directions}, under the standing hypothesis $n\ge 6$
and with the axiom profile just described.  So is
Proposition~\ref{prop:local}, whose proof above reduces it to the
twelve-tone exclusion.

\emph{How existence was formalized, and a remark on the proof.}  The
pumping induction of Section~\ref{sec:pumping} was \emph{not}
formalized as an induction.  Formalizing it directly would have
required the theorem that a balanced connected directed multigraph
admits an Eulerian circuit; a survey of Mathlib established that no
such result exists there in any form---its Eulerian material concerns
simple graphs, supplies only conditions necessary for a \emph{given}
Eulerian trail, and lists the converse as an open task---so that route
would have meant developing directed-multigraph walk theory and
Hierholzer's theorem as a prerequisite.

Instead the construction was made explicit.  In level coordinates the
entries of every form at level $j$ lie in $\{\pm1,\pm j,\pm(j-1)\}$,
independently of $n$; a step word is therefore a word in the level
indices, of the same shape for every temperament, and the effect of
pumping is a local splice.  Carrying this out yields closed formulas:
with $n=15+6t$,
\[
W_t \;=\; [\,2,\,-1,\,6\,]\ \Vert\ \beta(7)\,\Vert\,\beta(10)\,\Vert\cdots\Vert\,\beta(7{+}3(t{-}1))
\ \Vert\ [\,-(7{+}3t)\,]\ \Vert\ [\,-1,-3,1,-5,6,1,4,-3\,],
\]
where the inserted block is $\beta(L)=[-L,\,-1,\,L{+}2,\,1,\,-(L{+}1),\,L{+}2]$,
and a companion formula for $n=18+6t$.  Each block sums to $3$ while
the cap entry moves by $-3$, so the word closes at every $n$; the
chromatic zigzag is the fixed opening $[2,-1]$, which the splice never
touches.  What is verified in Lean is that these words are interval
cycles for all $t$: the pumping argument survives as the derivation of
the formulas, while the formal proof is a direct parametric
verification, free of graph theory.  Readers who prefer the
construction to the induction may take the displayed words as the
definition and Section~\ref{sec:pumping} as their motivation.

\emph{Availability.}  The development, together with the programs
reproducing the paper's enumerative claims, is deposited in a public
repository with a permanent identifier; continuous integration rebuilds
it from a clean clone, rejects any unproved placeholder or new axiom
declaration, and publishes the axiom audit as a build artifact.
\ifanon
\else
It is archived at
\[
\texttt{https://doi.org/10.5281/zenodo.22040738}
\]
and developed at
\texttt{https://github.com/DavidVFeldman/missing-zigzag-new} (release
\texttt{v1.1}); see \cite{Repo}.
\fi

\emph{Not verified.}  The censuses of Section~\ref{sec:census}, the
sector analysis of Section~\ref{sec:dynamics}, and the results of
Appendix~\ref{app:serial} (near-rows, all-trichord rings, the row
segue) are outside the formalization; they rest on the computations
described in the text.  Of the finite verifications in the main text,
the tallies of Proposition~\ref{prop:certificates} and the local
adjacency count of Proposition~\ref{prop:local} likewise rest on those
computations.

\end{document}